\documentclass[11pt]{amsart}

\usepackage{amssymb, amsfonts, amsmath, amsthm}
\usepackage{tikz}

\newtheorem{theorem}{Theorem}[section]
\newtheorem{lemma}[theorem]{Lemma}
\newtheorem{proposition}[theorem]{Proposition}
\newtheorem{corollary}[theorem]{Corollary}
\newtheorem{prop-and-def}[theorem]{Proposition and Definition}

\theoremstyle{definition}
\newtheorem{definition}[theorem]{Definition}
\newtheorem{notation}[theorem]{Notation}
\newtheorem{remark}[theorem]{Remark}
\newtheorem{defrem}[theorem]{Definition and Remark}

\newtheorem{notation-and-remark}[theorem]{Notation and Remark}
\newtheorem{remark-and-definition}[theorem]{Remark and Definition}
\newtheorem{example}[theorem]{Example}
\newtheorem{ad-hoc}[theorem]{ }
\numberwithin{equation}{section}

\newcommand{\cA}{ {\mathcal A} }

\newcommand{\bC}{ {\mathbb C} }

\newcommand{\fM}{ {\mathfrak M} }

\newcommand{\bN}{ {\mathbb N} }

\newcommand{\bZ}{ {\mathbb Z} }

\newcommand{\piodd}{ \pi^{(\mathrm{odd})} }
\newcommand{\rhoeven}{ \rho^{(\mathrm{even})} }

\newcommand{\calt}{ \mathrm{calt} }
\newcommand{\Cat}{ \mathrm{Cat} }

\newcommand{\depth}{ \mathrm{depth} }

\newcommand{\Int}{ \mbox{Int} }
\newcommand{\Kr}{ \mbox{Kr} }

\newcommand{\oddtuple}{ \mathrm{oddtuple} }

\newcommand{\OuterMax}{ \mathrm{OuterMax} }
\newcommand{\Parent}{ \mathrm{Parent} }
\newcommand{\Ran}{ \mbox{Ran} }

\newcommand{\ee}{\varepsilon}
\newcommand{\oneA}{ 1_{{ }_{\cA}} }
\newcommand{\NCacfriendly}{ NC_{\mathrm{ac-friendly}} }
\newcommand{\ecdef}{ \stackrel{def}{\Longleftrightarrow} }

\newcommand{\lnest}{ \stackrel{\mathrm{nest}}{<} }
\newcommand{\leqnest}{ \stackrel{\mathrm{nest}}{\leq} }
\newcommand{\gnest}{ \stackrel{\mathrm{nest}}{>} }
\newcommand{\geqnest}{ \stackrel{\mathrm{nest}}{\geq} }

\title[Use of Boolean cumulants for free random variables]{Using
Boolean cumulants to study multiplication and anticommutators of 
free random variables}

\author[M. Fevrier]{Maxime Fevrier}
\address{Maxime Fevrier: Laboratoire de Math\'ematiques d'Orsay,
Universit\'e Paris Sud, CNRS, 
Universit\'e Paris-Saclay, 91405 Orsay, France.}
\email{maxime.fevrier@u-psud.fr}
\author[M. Mastnak]{Mitja Mastnak}
\address{Mitja Mastnak: Department of Mathematics \& Computing Science, 
        Saint Mary's University, Halifax, Nova Scotia B3H 3C3, Canada.}
\email{mmastnak@cs.smu.ca}
\author[A. Nica]{Alexandru Nica}
\thanks{MM, AN: research supported by a Discovery Grant from 
	NSERC, Canada.}
\address{Alexandru Nica: Department of Pure Mathematics, 
	University of Waterloo, Ontario, Canada.}
\email{anica@uwaterloo.ca}
\author[K. Szpojankowski]{Kamil Szpojankowski}
\thanks{KSz: research partially suported by NCN grant 2016/23/D/ST1/01077}
\address{Kamil Szpojankowski:
	Faculty of Mathematics and Information Science,
	Warsaw University of Technology, Poland.}
\email{k.szpojankowski@mini.pw.edu.pl}

\begin{document}

\begin{abstract}
We study how Boolean cumulants can be used in order to address
operations with freely independent random variables, particularly 
in connection to the $*$-distribution of the product of two 
selfadjoint freely independent random variables, and in connection 
to the distribution of the anticommutator of such random variables.
\end{abstract}

\maketitle

\section{Introduction}

\subsection{Multiplication of free random variables, in 
terms of free cumulants.}

$\ $

\noindent
Let $( \cA , \varphi )$ be a noncommutative probability space.
It is known since the 90's (cf. \cite{NiSp-Vo1996}) how to 
handle the multiplication of two freely independent elements 
of $\cA$ in terms of free cumulants.  More precisely, let 
$( \kappa_n : \cA^n \to \bC )_{n=1}^{\infty}$ be the family 
of free cumulant functionals of $( \cA , \varphi )$.  If 
$a,b \in \cA$ are freely independent, then the free cumulants
of the product $ab$ are described by the formula
\begin{equation}    \label{eqn:11a}
\kappa_n (ab, \ldots , ab) =
\sum_{\pi \in NC(n)} \prod_{U \in \pi} \kappa_{|U|} (a, \ldots , a)
\cdot \prod_{V \in \Kr ( \pi )} \kappa_{|V|} (b, \ldots , b),
\end{equation}
where $NC(n)$ is the lattice of non-crossing partitions 
of $\{ 1, \ldots , n \}$, and
$\Kr : NC(n) \to NC(n)$ is an important anti-automorphism of 
this lattice, called Kreweras complementation map.  The formula
(\ref{eqn:11a}) is very useful because it allows one to take 
advantage of many pleasant properties the lattices 
$NC(n)$ are known to have.  In particular, upon re-writing 
(\ref{eqn:11a}) in terms of formal power series and upon doing 
suitable manipulations, one can use it (cf. \cite{NiSp1997}) 
to derive the multiplicativity of the well-known $S$-transform 
of Voiculescu \cite{Vo1987}.

In view of how we will make our presentation of results below, 
it is worth mentioning here that the clearest proof of the formula
(\ref{eqn:11a}) is made in 3 steps, as follows:

\begin{equation}    \label{eqn:11b}
\left\{
\begin{array}{ll}
\mbox{Step 1.}  & 
\mbox{On the left-hand side of (\ref{eqn:11a}), use the formula}   \\
                & \mbox{$\ $ (with summation over $NC(2n)$) which describes}  \\
                & \mbox{$\ $ free cumulants with products as arguments.}      \\
\mbox{Step 2.}  & \mbox{Use the fact that, due to the 
                        freeness of $a$ from $b$,}            \\
                & \mbox{$\ $ all their mixed free cumulants 
                         vanish.}                             \\
\mbox{Step 3.}  & \mbox{Perform a direct combinatorial 
                        analysis of the non-crossing}         \\
                & \mbox{$\ $ partitions in $NC(2n)$ which 
                        were not pruned in Step 2.}
\end{array}   \right.
\end{equation}  

\vspace{6pt}

Let us now upgrade to the framework where $( \cA , \varphi )$
is a $*$-probability space, and where $a,b$ are two freely independent 
selfadjoint elements of $\cA$.  Since $ab$ isn't generally selfadjoint,
we now need to keep track of the joint moments, or equivalently of the 
joint free cumulants of $ab$ and $(ab)^{*} = ba$.  That is, we now need 
to look at free cumulants of the form
\begin{equation}    \label{eqn:11c}
\kappa_n \bigl( \, (ab)^{\ee (1)}, \ldots , (ab)^{\ee (n)} 
\, \bigr) , \ \mbox{ with $n \in \bN$ and }
\ee = ( \ee (1), \ldots , \ee (n) ) \in \{ 1,* \}^n .  
\end{equation}    
The tools invoked in Steps 1 and 2 of (\ref{eqn:11b}) can still be 
used in connection to the cumulants from (\ref{eqn:11c}).  But the 
combinatorial analysis in Step 3 (where some version of the Kreweras 
complementation map would be hoped to appear) becomes ad-hoc and 
does not seem to reveal a pattern -- this is seen on very simple 
examples, e.g. when doing the calculation which expresses 
$\kappa_3 ( ab, ab, (ab)^{*} )$ in terms of the $*$-free cumulants 
of $a$ and those of $b$.

\vspace{10pt}

\subsection{Use Boolean cumulants instead of free cumulants?}

$\ $

\noindent
For a noncommutative probability space $( \cA , \varphi )$, 
one can also consider the family 
$( \beta_n : \cA^n \to \bC )_{n=1}^{\infty}$ of 
Boolean cumulant functionals of $( \cA , \varphi )$.
Boolean cumulants are the analogue of free cumulants in the parallel
(and simpler) world of Boolean probability.  They are known to have 
some direct connections with free probability, particularly in 
connection to a development called ``Boolean Bercovici-Pata bijection''.  
An intriguing fact around this topic is that
\begin{center}
``the Boolean Bercovici-Pata bijection preserves the structure 

behind multiplication of free random variables''.
\end{center}
A possible way to pitch this fact is as follows: when one describes 
the multiplication of two freely independent elements $a, b \in \cA$ 
in terms of Boolean cumulants, the resulting formula has 
exactly the same structure as in (\ref{eqn:11a}):
\begin{equation}    \label{eqn:12a}
\beta_n (ab, \ldots , ab) =
\sum_{\pi \in NC(n)} \prod_{U \in \pi} \beta_{|U|} (a, \ldots , a)
\cdot \prod_{V \in \Kr ( \pi )} \beta_{|V|} (b, \ldots , b),
\ \ \forall \, n \geq 1.
\end{equation}
The formula (\ref{eqn:12a}) was first found in
\cite[Theorem 2']{BeNi2008}.  It can be proved via a strategy with 
3 steps parallel to the one described in (\ref{eqn:11b}).  
(See also \cite[Lemma 3.2]{PoWa2011} for a similar result stated in
terms of the so-called ``c-free cumulants'', which relate at the 
same time to free and to Boolean cumulants.)

The main point of the present paper is that, for Boolean cumulants, 
the strategy with 3 steps can be pushed to the framework where 
$( \cA, \varphi )$ is a $*$-probability space and where we look at 
Boolean cumulants of the form
\begin{equation}    \label{eqn:12b}
\beta_n \bigl( \, (ab)^{\ee (1)}, \ldots , (ab)^{\ee (n)} \, \bigr), 
\mbox{ with $n \in \bN$ and }
\ee = ( \ee (1), \ldots , \ee (n) ) \in \{ 1,* \}^n .  
\end{equation}
What makes this possible is the use of an alternative facet of Kreweras 
complementation, suggested by the work in \cite{BeNi2008}.  This facet 
of Kreweras complementation is discussed in the next subsection.
(We reiterate here that the possibility of using it is specific to 
Boolean cumulants, and -- as seen on very simple examples -- fails
to work for free cumulants.)

\vspace{10pt}

\subsection{Kreweras complementation and VNRP property.}

$\ $

\noindent
The definition of the Kreweras complement for a 
partition $\pi \in NC(n)$, as originally given by 
G. Kreweras \cite{Kr1972}, goes via a maximality argument.
One considers partitions (not necessarily non-crossing) 
of $\{ 1, \ldots , 2n \}$ of the form 
``$\piodd \sqcup \rhoeven$'', which have a copy of
$\pi$ placed on $\{ 1,3, \ldots , 2n-1 \}$ and a copy of 
some other partition $\rho \in NC(n)$ placed on 
$\{ 2,4, \ldots , 2n \}$.  The Kreweras complement 
$\Kr_n ( \pi )$ is the maximal element, with respect 
to the reverse refinement order ``$\leq$'' on 
$NC(n)$, for the set 
$\{ \rho \in NC(n) \mid \piodd \sqcup \rhoeven \in NC(2n) \}$.

The paper \cite{BeNi2008} considered another partial order 
``$\ll$'' on $NC(n)$, coarser than the reverse refinement order
$\leq$, which is useful for studying connections between free 
cumulants and Boolean cumulants.  In Proposition 6.10 of 
\cite{BeNi2008}, another maximality property related to Kreweras 
complements was noticed to hold: the partitions of 
the form $\piodd \sqcup ( \Kr_n (\pi) )^{\mathrm{(even)}}$ are 
the maximal elements with respect to $\ll$ for the set 
\begin{equation}   \label{eqn:13a}
\Bigl\{ \sigma \in NC(2n)  \begin{array}{ll}
\vline & \mbox{every block of $\sigma$ is contained either}  \\
\vline & \mbox{in $\{1,3, \ldots , 2n-1 \}$
               or in $\{ 2,4, \ldots , 2n \}$, }            \\
\vline & \mbox{and $\sigma$ has exactly two outer blocks}
\end{array}  \Bigr\} ;
\end{equation}
moreover, for every $\sigma$ in the above set, 
there exists a unique $\pi \in NC(n)$ such that 
$\sigma \ll \piodd \sqcup ( \Kr_n (\pi))^{\mathrm{(even)}}$.

\noindent
[The concept of outer block, and the related concept
of depth for a block of a non-crossing partition are reviewed
in Section 2 below.  
The requirement ``$\sigma$ has exactly two outer blocks''
in (\ref{eqn:13a}) is a minimality condition, since 
the block of $\sigma$ which contains the number $1$ and the 
block of $\sigma$ which contains the number $2n$ are always 
sure to be outer blocks.]

\vspace{6pt}

In the present paper we extend the result described above
to a framework where considering parities is a special case 
of considering a colouring 
$c : \{ 1, \ldots , 2n \} \to \{ 1,2 \}$ (given by 
$c(i) = i ( \mbox{mod $2$})$ for $1 \leq i \leq 2n$).
Upon examining the point of view of colourings, one finds 
that the property which isolates partitions of the form 
$\piodd \sqcup ( \Kr_n (\pi))^{\mathrm{(even)}}$ within 
the set (\ref{eqn:13a}) is a certain 
{\em vertical-no-repeat property}, or VNRP for short.  It 
is straightforward how to define VNRP for a general colouring 
(Definition \ref{def:41} below).  Given a partition 
$\sigma \in NC(m)$ and a colouring in $s$ colours 
$c : \{ 1, \ldots , m \} \to \{ 1, \ldots , s \}$ which is 
constant along the blocks of $\sigma$, the fact that 
$\sigma$ and $c$ have VNRP amounts to the requirement that
\[
c ( \, \Parent_{\sigma} (V) \, ) \neq c(V),
\ \ \mbox{ for every inner block $V$ of $\sigma$.}
\]
[Here we refer to the fairly intuitive fact that every inner 
block $V$ of a non-crossing partition $\sigma$ must have 
a ``parent-block'' $\Parent_{\sigma} (V)$ into which it is 
nested.  The precise definition of how this goes is 
reviewed in Section 2 below.]

The key-property of VNRP, extending the considerations on 
Kreweras complements from the preceding paragraph, is then 
stated as follows.

\begin{theorem}   \label{thm:131}
{\em (Key-property of VNRP.) }

\noindent
Let $m$ be in $\bN$, let 
$c : \{ 1, \ldots , m \} \to \{ 1, \ldots , s \}$
be a colouring, and consider the set of partitions 
\[
NC(m;c) := \{ \sigma \in NC(m) \mid 
\mbox{ $c$ is constant on every block of $\sigma$} \} .
\]
For every $\sigma \in NC(m;c)$ there exists a $\tau \in NC(m;c)$,
uniquely determined, such that $\sigma \ll \tau$ and such that 
$\tau$ has the VNRP property 
\footnote{ The correct formulation here would be to say that 
$\tau$ and $c$ (together) have VNRP.  We will occasionally replace
this with saying that ``$\tau$ has VNRP with respect to $c$'',
or that ``$c$ has VNRP with respect to $\tau$''.}
with respect to $c$.
\end{theorem}

$\ $

\subsection{Free independence in terms of Boolean cumulants.}

$\ $

\noindent
An easy calculation based on Theorem \ref{thm:131} leads to a
description of free independence in terms of Boolean cumulants,
as follows.

\begin{theorem}      \label{thm:141}
Let $( \cA , \varphi )$ be a noncommutative probability
space and let $( \beta_n : \cA^n \to \bC )_{n=1}^{\infty}$ be 
the family of Boolean cumulant functionals associated to it.  
Let $\cA_1,\ldots,\cA_s\subseteq \cA$ be unital subalgebras.
The following two statements are equivalent.

\vspace{6pt}

(1) $\cA_1,\ldots,\cA_s$ are free with respect to $\varphi$.

\vspace{6pt}

(2) For every $n \in \bN$, every colouring 
$c : \{ 1, \ldots , n \} \to \{ 1, \ldots , s \}$, and every 
$a_1 \in \cA_{c(1)}, \ldots a_n \in \cA_{c(n)}$, one has
\begin{equation}    \label{eqn:141a}
\beta_n (a_1, \ldots , a_n) = 
\sum_{  \begin{array}{c}
{\scriptstyle \pi \in NC(n;c) \ with}  \\
{\scriptstyle unique \ outer \ block}   \\
{\scriptstyle and \ with \ VNRP}
\end{array}  } \ \prod_{V \in \pi} 
\beta_{|V|} \bigl( (a_1, \ldots a_n) | V \bigr).
\end{equation}
\end{theorem}

\begin{remark}   \label{rem:142}
(1) The essential part of Theorem \ref{thm:141} is the 
implication $(1) \Rightarrow (2)$, which gives an explicit 
formula for Boolean cumulants with free arguments.
Once this is established, the converse
$(2) \Rightarrow (1)$ follows by combining
$(1) \Rightarrow (2)$ with a standard ``replica trick''.

(2) Coming from the study of a very general notion of 
noncommutative independence, Proposition 4.30 of the recent 
paper \cite{JeLi2019} gives a description of free independence
in terms of Boolean cumulants which (when considered with 
$\bC$ as field of scalars) is equivalent to the above 
Theorem \ref{thm:141}.  To be precise, condition 2 in 
Proposition 4.30 of \cite{JeLi2019} is the moment formula 
which comes out when one performs an additional summation 
over interval partitions on both sides of (\ref{eqn:141a})
-- see Corollary \ref{cor:53} below.  Conversely, the latter 
moment formula can be used to retrieve Equation (\ref{eqn:141a}), 
via an easy application of M\"obius inversion.

(3) Equation (\ref{eqn:141a}) implies an amusing formula 
for the Boolean cumulants of the sum of two freely independent 
random variables.  The structure of this formula cannot be as 
simple as what one gets by using free cumulants, but we present 
it nevertheless in Proposition \ref{prop:55} below, in 
anticipation of the similarly looking formula concerning free 
anticommutators (where the use of free cumulants does not 
provide a simpler alternative).
\end{remark}

\vspace{10pt}

\subsection{Joint Boolean cumulants for 
$\mathbf{ab}$ and $\mathbf{(ab)^{*}}$, and
free anticommutators.}

$\ $

\noindent
We now continue the thread from Section 1.2, concerning the 
joint Boolean cumulants of $ab$ and $(ab)^{*}$, where $a$ 
and $b$ are freely independent selfadjoint elements in a 
$*$-probability space.  We will put into evidence some
special sets of non-crossing partitions which appear in the 
explicit formula for a joint cumulant
$ \beta_n \bigl( \, (ab)^{\ee (1)}, \ldots , 
(ab)^{\ee (n)} \, \bigr)$ as in (\ref{eqn:12b}) and then
(upon doing a summation over $\ee \in \{ 1,* \}^n$) in the 
explicit formula for the Boolean cumulants of a free 
anticommutator $ab + ba$.  We will refer to these special 
non-crossing partitions by using the ad-hoc term of 
``anticommutator-friendly'', due to how they appear in the 
formula for the Boolean cumulants of a free anticommutator, 
in Theorem \ref{thm:154} below.  Their actual definition 
doesn't, however, require any knowledge of a free 
probabilistic framework, and is stated as follows.

$\ $

\begin{definition}    \label{def:151}
Let $n$ be a positive integer, let $\sigma$ be a 
partition in $NC(2n)$, and let us consider the set 
$\OuterMax ( \sigma ) := \{ \max (W) \mid
\mbox{ $W$ is an outer block of $\sigma$} \}$.
We will say that $\sigma$ is 
{\em anticommutator-friendly} when it satisfies 
the following two conditions.

\vspace{6pt}

\noindent
(AC-Friendly1)  $\OuterMax ( \sigma ) \subseteq
\{ 1,3, \ldots , 2n-1 \} \cup \{ 2n \}$.

\vspace{6pt}

\noindent
(AC-Friendly2)  For every 
$j \in \{ 1,3, \ldots , 2n-1 \} \setminus \OuterMax ( \sigma )$, 
one has $\depth_{\sigma} (j) \neq \depth_{\sigma} (j+1)$,
where ``$\depth_{\sigma} (j)$'' stands for the depth 
of the block of $\sigma$ which contains the number $j$.
\end{definition}

\begin{notation-and-remark}   \label{def:152}
For every $n \in \bN$ we will denote
\[
\NCacfriendly (2n) := \{ \sigma \in NC(2n) \mid
\mbox{ $\sigma$ is anticommutator-friendly} \} .
\]
For instance $\NCacfriendly (2)$ consists of only one 
partition, $\{ \, \{ 1 \} , \{ 2 \} \} \in NC(2)$.
(The partition $\sigma = \{ \, \{ 1,2 \} \, \}$ is not 
anticommutator-friendly because it has 
$\depth_{\sigma} (1) = \depth_{\sigma} (2) = 0$.)
The next such set, $\NCacfriendly (4)$, consists of 
the 5 partitions which are depicted in Figure 1 below.
The cardinalities of the sets of partitions 
$\NCacfriendly (2n)$ are tractable, in the respect 
that their generating series satisfies an algebraic 
equation of order 4, which can be solved explicitly:
\begin{equation}   \label{eqn:15a}
\sum_{n=1}^{\infty}
| \, \NCacfriendly (2n) \, | z^n 
= \frac{1}{2} 
  - \sqrt{ (1-8z) \frac{1-2z-\sqrt{1-8z}}{8z} }.
\end{equation}
\end{notation-and-remark} 

\vspace{0.5cm}

\begin{minipage}{1\linewidth}
\begin{center}
\begin{tikzpicture}			
		\draw[thick] (0,0) -- (0,0.25);
		\draw[thick] (2,0) -- (2,0.25);
		\draw[thick] (-1,0) to[out=90,in=90] node [sloped,above] {} (1,0);
		
		\node[below] (1) at (-1,0) {$1$};
		\node[below] (2) at (0,0) {$2$};
		\node[below] (3) at (1,0) {$3$};
		\node[below] (4) at (2,0) {$4,$};

	\draw[thick] (4,0) -- (4,0.25);
	\draw[thick] (7,0) -- (7,0.25);
	\draw[thick] (5,0) to[out=90,in=90] node [sloped,above] {} (6,0);
	
	\node[below] (1) at (4,0) {$1$};
	\node[below] (2) at (5,0) {$2$};
	\node[below] (3) at (6,0) {$3$};
	\node[below] (4) at (7,0) {$4,$};
		
		\draw[thick] (9,0) -- (9,0.25);
		\draw[thick] (11,0) -- (11,0.25);
		\draw[thick] (10,0) to[out=90,in=90] node [sloped,above] {} (12,0);

		\node[below] (7) at (9,0) {$1$};
		\node[below] (8) at (10,0) {$2$};
		\node[below] (9) at (11,0) {$3$};
		\node[below] (10) at (12,0) {$4,$};
\end{tikzpicture}
\end{center}
\end{minipage}
\begin{minipage}{1\linewidth}
\begin{center}
	\begin{tikzpicture}			
	\draw[thick] (2,0) -- (2,0.25);
	\draw[thick] (3,0) -- (3,0.25);
	\draw[thick] (1,0) to[out=90,in=90] node [sloped,above] {} (4,0);
	
	\node[below] (1) at (1,0) {$1$};
	\node[below] (2) at (2,0) {$2$};
	\node[below] (3) at (3,0) {$3$};
	\node[below] (4) at (4,0) {$4,$};

	\draw[thick] (7,0) to[out=90,in=90] node [sloped,above] {} (10,0);
	\draw[thick] (8,0) to[out=90,in=90] node [sloped,above] {} (9,0);

	\node[below] (7) at (7,0) {$1$};
	\node[below] (8) at (8,0) {$2$};
	\node[below] (9) at (9,0) {$3$};
	\node[below] (10) at (10,0) {$4.$};
	\end{tikzpicture}
\end{center}
\end{minipage}

\begin{center}
{\bf Figure 1.}
{\em The 5 partitions in $\NCacfriendly (4)$.}
\end{center}

$\ $

The formulas to be stated in Theorems  \ref{thm:153} 
and \ref{thm:154} below will also refer to a 
``canonical alternating colouring'' of the blocks of 
a non-crossing partition, which is described next.

\begin{definition}    \label{def:153}
Let $m \in \bN$ and $\sigma \in NC(m)$ be given.  We 
will use the name {\em canonical alternating colouring} 
of $\sigma$ for the colouring
$\calt_{\sigma} : \{ 1, \ldots , m \} \to \{ 1,2 \}$
whose values on the blocks of $\sigma$ are determined
by the following conditions:

\vspace{6pt}

\noindent
(C-Alt1) Denoting by $W_1$ the block of $\sigma$ 
which contains the number $1$, one has
$\calt_{\sigma} (W_1) = 1$.

\vspace{6pt}

\noindent
(C-Alt2) If $W$ and $W'$ are ``consecutive'' outer 
blocks of $\sigma$, with $\min (W') = 1 + \max (W)$, then 
$\calt_{\sigma} (W') \neq \calt_{\sigma} (W)$.

\vspace{6pt}

\noindent
(C-Alt3) If $V$ is an inner block of $\sigma$, then
$\calt_{\sigma} (V) \neq 
\calt_{\sigma} ( \, \Parent_{\sigma} (V) \, )$.
\end{definition}

$\ $

Note that if $\sigma \in NC(m)$ has a unique outer block, 
then $\calt_{\sigma}$ simply follows the parities of 
the depths of blocks of $\sigma$.  For a general 
$\sigma \in NC(m)$, $\calt_{\sigma}$ first does an 
alternating colouring of the outer blocks of $\sigma$, 
going from left to right; then for every outer block 
$W$ of $\sigma$ one follows the vertical alternance idea 
in order to colour the blocks of $\sigma$ which are 
nested inside $W$.

In order to state the explicit formula for a joint
Boolean cumulant of the kind indicated in (\ref{eqn:12b}) 
of Section 1.2, there is one last observation we 
need to make, namely that: in the canonical alternating 
colouring of a partition $\sigma \in \NCacfriendly (2n)$ 
one can naturally read (encoded in the colouring) a tuple 
$\ee \in \{ 1,* \}^n$, which will be denoted as 
``$\oddtuple ( \sigma )$''  -- see Notation \ref{def:64} 
below for the precise definition.
We then have the following theorem.

$\ $

\begin{theorem}   \label{thm:153}
Let $( \cA , \varphi )$ be a $*$-probability space 
and let $( \beta_n : \cA^n \to \bC )_{n=1}^{\infty}$ be the 
family of Boolean cumulant functionals associated to it.
Consider two selfadjoint elements $a,b \in \cA$ such that 
$a$ is freely independent from $b$, and consider the 
sequences of Boolean cumulants
$( \beta_n (a) )_{n=1}^{\infty}$ and 
$( \beta_n (b) )_{n=1}^{\infty}$
of $a$ and of $b$ (where we use natural abbreviations such 
as $\beta_n (a) := \beta_n (a, \ldots , a)$, $n \in \bN$).

\vspace{6pt}

(1) For $n \in \bN$ and 
$\ee = ( \ee (1), \ldots , \ee (n)) \in \{ 1,* \}^n$
such that $\ee (1) = 1$, one has
\begin{equation}   \label{eqn:153a}
\beta_n \bigl( \, (ab)^{\ee (1)}, \ldots ,
(ab)^{\ee (n)} \, \bigr) =
\end{equation}
\[
\sum_{ \begin{array}{c}
{\scriptstyle \sigma \in \NCacfriendly (2n),} \\
{\scriptstyle such \ that}    \\
{\scriptstyle \oddtuple ( \sigma ) = \ee}
\end{array} } 
\ \Bigl( 
\prod_{ \begin{array}{c}
{\scriptstyle U \in \sigma , \ with}  \\
{\scriptstyle \calt_{\sigma} (U) = 1} 
\end{array} } \
\beta_{|U|} (a) \Bigr)  \cdot
\Bigl( \prod_{ \begin{array}{c}
{\scriptstyle V \in \sigma , \ with}  \\
{\scriptstyle \calt_{\sigma} (V) = 2} 
\end{array} } \ \beta_{|V|} (b) \Bigr). 
\]

\vspace{6pt}

(2) Let $n \in \bN$ and let
$\ee = ( \ee (1), \ldots , \ee (n)) \in \{ 1,* \}^n$
be such that $\ee (1) = *$.  Consider the complementary 
tuple $\ee ' \in \{ 1,* \}^n$, uniquely determined by 
the requirement that $\ee ' (i) \neq \ee (i)$, for all
$1 \leq i \leq n$.  One has
\begin{equation}   \label{eqn:153b}
\beta_n \bigl( \, (ab)^{\ee (1)}, \ldots ,
(ab)^{\ee (n)} \, \bigr) =
\end{equation}
\[
\sum_{ \begin{array}{c}
{\scriptstyle \sigma \in \NCacfriendly (2n),} \\
{\scriptstyle such \ that}    \\
{\scriptstyle \oddtuple ( \sigma ) = \ee '}
\end{array} } 
\ \Bigl(   \prod_{ \begin{array}{c}
{\scriptstyle U \in \sigma, \ with}  \\
{\scriptstyle \calt_{\sigma} (U) = 1} 
\end{array} } \
\beta_{|U|} (b) \Bigr)  \cdot
\Bigl(   \prod_{ \begin{array}{c}
{\scriptstyle V \in \sigma , \ with}  \\
{\scriptstyle \calt_{\sigma} (V) = 2} 
\end{array} } \ \beta_{|V|} (a) \Bigr). 
\]
\end{theorem}

$\ $

By summing over $\ee \in \{ 1,* \}^n$ in Theorem \ref{thm:153},
we arrive to a formula for the Boolean cumulants of a free
anticommutator.

$\ $

\begin{theorem}    \label{thm:154}
Consider the same framework and notation as in 
Theorem \ref{thm:153}.  Then, for every $n \in \bN$, the 
$n$-th Boolean cumulant of $ab + ba$ is 
\[
\beta_n (ab+ba) = 
\sum_{\sigma \in \NCacfriendly (2n)} \ 
\Bigl( \prod_{   \begin{array}{c}
{\scriptstyle U \in \sigma ,   }  \\
{\scriptstyle \calt_{\sigma} (U) = 1} 
\end{array} } \ \beta_{|U|} (a) \Bigr)  \cdot
\Bigl( \prod_{   \begin{array}{c}
{\scriptstyle V \in \sigma ,   }  \\
{\scriptstyle \calt_{\sigma} (V) = 2} 
\end{array} } \ \beta_{|V|} (b) \Bigr) 
\] 
\begin{equation}   \label{eqn:154a}
+ \ \sum_{\sigma \in \NCacfriendly (2n)} \ 
\Bigl( \prod_{   \begin{array}{c}
{\scriptstyle U \in \sigma ,   }  \\
{\scriptstyle \calt_{\sigma} (U) = 1} 
\end{array} } \ \beta_{|U|} (b) \Bigr)  \cdot
\Bigl( \prod_{   \begin{array}{c}
{\scriptstyle V \in \sigma ,   }  \\
{\scriptstyle \calt_{\sigma} (V) = 2} 
\end{array} } \ \beta_{|V|} (a) \Bigr) .
\end{equation}
\end{theorem}

\begin{remark}   \label{rem:155}

(1) In Theorem \ref{thm:153} it should be noted that
if we make $\ee = (1,1, \ldots , 1)$, then what comes 
out is precisely the formula from \cite{BeNi2008} which
was reviewed in Equation (\ref{eqn:12a}).  A discussion 
of why this is the case appears in Remark \ref{rem:58} 
below.

(2) In Theorem \ref{thm:154}, we note that formula 
(\ref{eqn:154a}) simplifies a lot when $a$ and $b$ have 
the same distribution.  In 
this case, denoting by $( \lambda_n )_{n=1}^{\infty}$ 
the common sequence of Boolean cumulants of $a$ and of 
$b$, we find that the $n$-th Boolean cumulant of $ab+ba$ 
is 
\begin{equation}  \label{eqn:155a}
\beta_n (ab+ba) = 
2 \cdot \sum_{\sigma \in \NCacfriendly (2n)} \ 
\prod_{V \in \sigma} \lambda_{|V|}.
\end{equation} 
\end{remark}

\begin{remark}   \label{rem:156}
In order to put things into perspective, we give here some 
background on the past work around the problem of the free 
anticommutator.  A noteworthy fact to begin with is that 
this problem is vastly simplified when we make the additional 
hypothesis that $a$ and $b$ have symmetric distributions
(that is, $\varphi (a^{2n-1}) = 0 =\varphi (b^{2n-1})$ 
for all $n \in \bN$).  In this case, the nonselfadjoint 
element $ab \in \cA$ has a certain ``$R$-diagonal'' property 
(cf. Lecture 15 of \cite{NiSp2006}).  From here it follows 
in particular that $ab + ba$ and $i(ab-ba)$ have the same 
distribution (due to radial symmetry displayed by
$R$-diagonal elements); moreover, the common distribution 
of $ab + ba$ and $i(ab-ba)$ is tractable due to the very 
special form of joint free cumulants that an $R$-diagonal 
element and its adjoint are known to have.

It is interesting that the combinatorial study of $i(ab-ba)$ 
remains tractable even when we drop the assumption of 
$a$ and $b$ having symmetric distributions, because one can 
follow (cf. \cite{NiSp1998}) how terms cancel in the expansions 
of the free cumulants $\kappa_n ( i(ab-ba))$.  The situation 
is not at all the same concerning $ab + ba$, where there are 
no cancellations to be followed.  Here the expansions for 
moments or cumulants just create some large summations, and 
the combinatorial line of attack goes via precise 
identification of the combinatorial structures which appear 
as index sets for these large summations.

Another noteworthy possibility to be mentioned is 
via approaches that are plainly analytic in nature, and produce 
systems of equations which can in principle be used to 
calculate the Cauchy transform of $ab+ba$.  Such a system of 
equations is proposed in \cite{Va2003}.  Another possibility of 
proceeding on these lines is suggested by the linearization 
method championed in \cite{HeMaSp2018}.
\end{remark}

\vspace{10pt}

\subsection{Equations with $\boldmath{\eta}$-series.}

$\ $

\noindent
The possibility of approaching the distribution of $ab + ba$ 
via a system of equations in (not necessarily convergent)
power series can also be pursued in the framework of the 
present paper.  Here we use the generating power series
for Boolean cumulants, which are also known as 
{\em $\eta$-series}: for $a \in \cA$, we put 
\begin{equation}   \label{eqn:16a}
\eta_a (z) := \sum_{n=1}^{\infty} \beta_n (a) z^n
\in \bC [[z]].  
\end{equation}
In Section 6 of the paper we make a detailed analysis of the 
$\eta$-series $\eta_{ab+ba}$ of a free anticommutator, and 
we come up with a system of equations which, when solved, 
leads to the explicit determination of this $\eta$-series.  
Our derivation of this system of equations is combinatorial 
in nature, and is intimately related to the study of recursions 
satisfied by ac-friendly non-crossing partitions.  

The best way to describe our system of equations leading to 
$\eta_{ab+ba}$ is in a $2 \times 2$ matrix form, where we make 
use of some auxiliary power series 
$f_{a,a}, f_{a,a^{*}}, f_{a^{*},a}, f_{a^{*},a^{*}}$
grouped 
\footnote{It is convenient to have the entries of $F_a$ indexed 
by symbols $a$ and $a^{*}$, even though the intended use of $F_a$ 
is when $a$ is selfadjoint.  The rationale for this notation is 
given at the beginning of Section 6.1.}
in a matrix
\begin{equation}   \label{eqn:16b}
F_a  =\begin{bmatrix}
f_{a,a} & f_{a,a^*} \\
f_{a^*,a} & f_{a^*,a^*}
\end{bmatrix},
\end{equation}
and of some power series 
$f_{b,b}, \ldots , f_{b^{*},b^{*}}$ likewise grouped in a 
$2 \times 2$ matrix $F_b$.  For illustration, in this 
Introduction we present the special case when $a$ and $b$ 
have the same distribution.  In this case we only need to 
refer to the matrix $F_a$, and we have the theorem stated next.  
(In the case when $a$ and $b$ are not required to have the 
same distribution, we get a more involved system of equations, 
where we use both matrices $F_a$ and $F_b$.  This is described 
in Theorem \ref{thm:61} below.)

\begin{theorem}   \label{thm:161}
Let $( \cA , \varphi )$ be a $*$-probability space, and let
$a,b$ be selfadjoint elements of $\cA$ such that $a$ is free
from $b$ and such that $a,b$ have the same distribution.

(1) The matrix $F_a$ from Equation (\ref{eqn:16b}) is obtained 
by solving the matrix equation 
\begin{equation}   \label{eqn:161a}
F_a H_a = \eta_a (z H_a),
\end{equation}
where $\eta_a$ is the $\eta$-series of $a$ (as in (\ref{eqn:16a}), and
\begin{align*}
H_a & := \begin{bmatrix}
f_{a^*,a^*} (1-f_{a,a^*})^{-1} & f_{a^*,a}+f_{a^*,a^*} (1-f_{a,a^*})^{-1} f_{a,a}\\ 
(1-f_{a,a^*})^{-1} & (1-f_{a,a^*})^{-1} f_{a,a} 
\end{bmatrix}.
\end{align*}

(2) The $\eta$-series of $ab+ba$ can be obtained from the entries
of $F_a$ via the equation
\begin{equation}   \label{eqn:161b}
\eta_{ab+ba}(z^2) = 2 \bigl( f_{a,a^*}(z)+
\frac{f_{a,a}(z)f_{a^*,a^*}(z)}{1-f_{a^*,a}(z)} \bigr).
\end{equation}
\end{theorem}

\begin{remark}  \label{rem:162}
Very much in agreement with the discussion at the beginning of 
Remark \ref{rem:156}, the study of free anticommutators via 
equations in $\eta$-series also simplifies substantially 
in the case when $a$ and $b$ have symmetric distributions.  
In this case the matrices $F_a$ and $F_b$ mentioned above are 
sure to have some vanishing entries 
($f_{a,a} = f_{a^{*}, a^{*}} = 0$ and 
$f_{b,b} = f_{b^{*}, b^{*}} = 0$), and the systems of equations
that have to be solved become simpler, as shown in Proposition 
\ref{prop:64} and Corollary \ref{cor:65}.  

While the examples of free anticommutators of symmetric distributions 
are covered by the methods from \cite{NiSp1998}, it is 
nevertheless interesting to work out some examples of this kind
and combine them with a use ``in reverse'' of Theorem \ref{thm:154},
in order to obtain corollaries about the enumeration of ac-friendly
non-crossing partitions.  For instance, in order to count the 
non-crossing partitions $\sigma \in \NCacfriendly (2n)$ with the 
property that all blocks $V \in \sigma$ have even cardinality, one uses
elements $a,b \in \cA$ which are freely independent and have distribution 
$\frac{1}{4} ( \delta_{- \sqrt{2}} + \delta_{\sqrt{2}} ) 
+ \frac{1}{2} \delta_0$.
The reason for choosing the latter distribution is that the common
sequence $( \lambda_n )_{n=1}^{\infty}$ of Boolean cumulants for $a$ 
and $b$ simply has $\lambda_n = 1$ for $n$ even and $\lambda_n = 0$ 
for $n$ odd.  In view of Remark \ref{rem:155}(2), the Boolean cumulant 
$\beta_n (ab+ba)$ is then equal to twice the cardinality we are
interested to determine.  Upon combining this with the explicit 
formula obtained for $\eta_{ab+ba}$, we can determine
precisely what is the required cardinality, as explained in 
Example \ref{ex:69} and Corollary \ref{cor:610} below.
\end{remark}

\begin{example}   \label{ex:163}
In the framework of Theorem \ref{thm:161}, it is instructive
to consider the simplest possible non-symmetric example, where 
both $a$ and $b$ have distribution 
$\frac{1}{2} ( \delta_0 + \delta_2 )$.

\begin{center}
\includegraphics[width=0.8\textwidth]{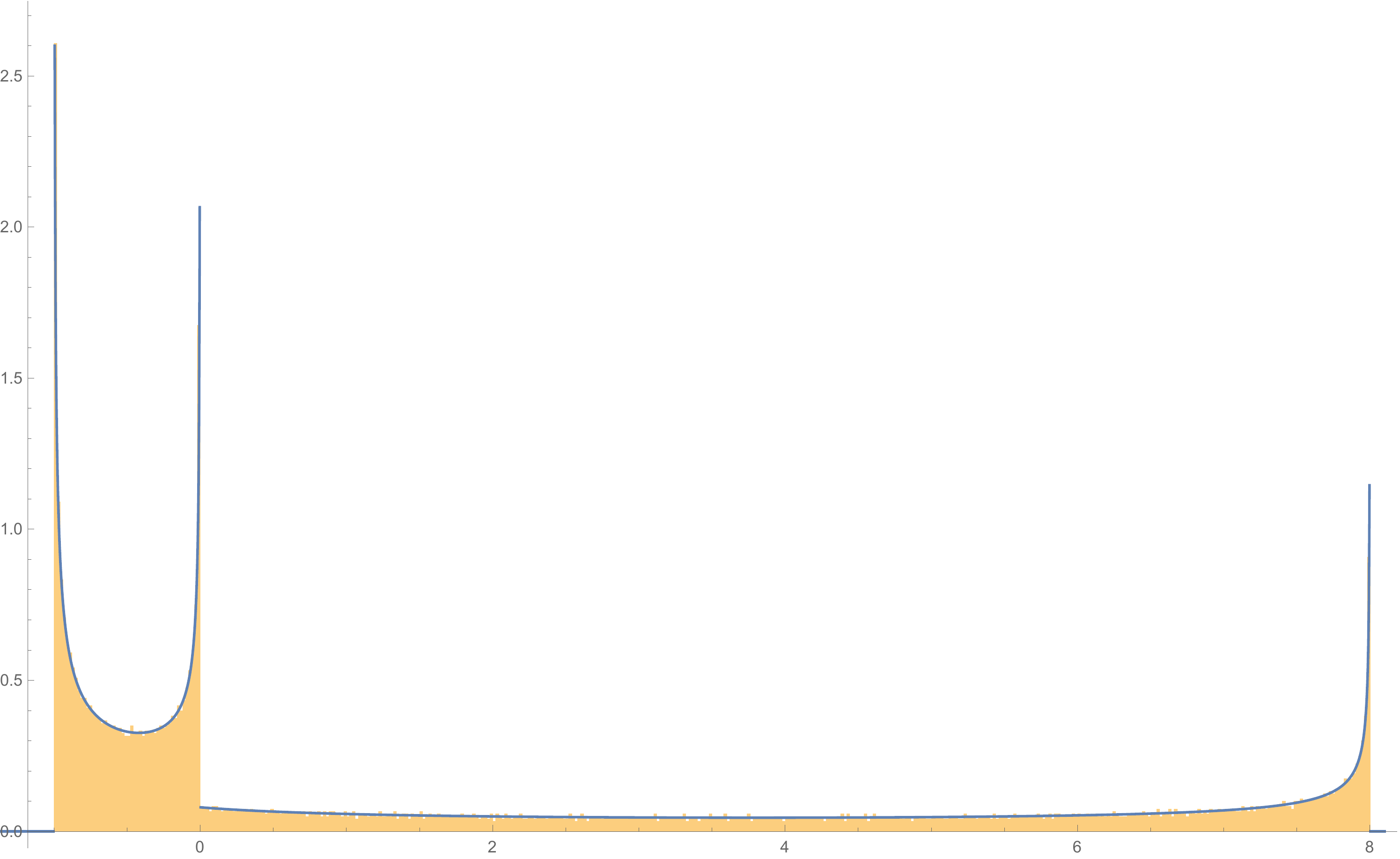}

{\bf Figure 2.}
{\em Plot of the density of distribution of $ab+ba$ 

for $a,b$ free and having distribution 
$\tfrac{1}{2}\delta_0+\tfrac{1}{2}\delta_2$,
together with

a histogram of eigenvalues of random matrix approximation.}
\end{center}

$\ $

The matrix equation from Theorem \ref{thm:161} is easy to 
solve in this example, and we end with an explicit formula for 
the $\eta$--series $\eta_{ab+ba} (z)$.  This can be followed 
with a calculation of Cauchy transform and with a Stieltjes 
inversion, in order to concretely determine what is the law 
of $ab+ba$ -- we find an absolutely continuous distribution 
supported on the interval $[-1,8]$.  The explicit formula for 
the density of this distribution and the calculations leading 
to it are presented in Proposition \ref{prop:611} below.  

Figure 2 shows the graph of the density $f(x)$ found 
for the law of $ab+ba$.  For a check, Figure 2 also shows a 
histogram of empirical eigenvalues distribution for $AB+BA$ 
where $A$ is a diagonal $6000\times 6000$ 
matrix with half of diagonal entries equal $0$ and half equal 2, 
and $B=UAU^*$ where $U$ is a random unitary matrix.

This example offers a very good illustration of how one gets 
to have different distributions for the free commutator and 
anticommutator -- indeed, the law of the commutator $i(ab-ba)$ 
is easily found to be the arcsine distribution on $[-2,2]$ 
(cf. Example \ref{ex:66}, and the discussion in the paragraph 
preceding Proposition \ref{prop:611}).
We point out that this example has a combinatorial significance 
as well, and can be used (cf. Corollary \ref{cor:612}) to infer 
the formula indicated in Equation (\ref{eqn:15a}) for the 
generating series of cardinalities of sets $NC_{ac-friendly} (2n)$.
\end{example}

\vspace{10pt}

\subsection{Organization of the paper.}

$\ $

\noindent
Besides the present Introduction, we have five other 
sections.  After a review of background in Section 2, we 
discuss VNRP and prove Theorem \ref{thm:131} in 
Section 3.  In Section 4 we discuss the applications
of VNRP to free independence via Boolean cumulants.
In Section 5 we prove the results about the joint 
Boolean cumulants of $ab$ and $(ab)^{*}$ and about the
Boolean cumulants of the free anticommutator which were 
advertised in Section 1.5 above.  Finally, in Section 6 
we consider the conversion from Boolean cumulants to 
$\eta$-series, and prove the results that were advertised
in Section 1.6 above.

$\ $

\section{Background and Notation}

In this section we review some background on set-partitions, and 
the two types of cumulants we want to work with.

\vspace{6pt}

\subsection{Nestings and depths for blocks of a non-crossing partition.}

$\ $

\noindent
We start by reviewing, for the sake of setting notation, the definition 
of the two basic types of set-partitions used in this paper, the 
{\em non-crossing partitions} and the {\em interval partitions}.

\begin{definition}   \label{def:21}
(1) Let $n$ be a positive integer and let 
$\pi = \{ V_1 , \ldots , V_k \}$ be a partition of 
$\{ 1, \ldots ,n \}$; that is, $V_1 , \ldots , V_k$ are non-empty
pairwise disjoint sets (called the {\em blocks} of $\pi$) with
$V_1 \cup \cdots \cup V_k$ = $\{ 1, \ldots , n \}$.  
The number $k$ of blocks of $\pi$ will be denoted as $| \pi |$,
and we will occasionally use the notation ``$V \in \pi$'' to mean
that $V$ is one of $V_1, \ldots , V_k$.

We say that $\pi \in NC(n)$ is an {\em interval partition} to 
mean that every block $V$ of $\pi$ is of the form
$V = [i,j] \cap \bN$ for some $1 \leq i \leq j \leq n$. 

We say that $\pi$ is a {\em non-crossing partition} to mean that 
for every $1 \leq i_1 < i_2 < i_3 < i_4 \leq n$ such that $i_1$ 
is in the same block with $i_3$ and $i_2$ is in the same block 
with $i_4$, it necessarily follows that all of $i_1, \ldots , i_4$ 
are in the same block of $\pi$.  

\vspace{6pt}

(2) For every $n \in \bN$, we denote by $\Int (n)$ the set of all
interval partitions of $\{ 1, \ldots , n \}$, and we denote by 
$NC(n)$ the set of all non-crossing partitions of $\{ 1, \ldots , n \}$.  
\end{definition}

\begin{remark}   \label{rem:21}
Clearly, one has $\Int (n) \subseteq NC(n)$ for all
$n \in \bN$.  It is not hard to see that 
$| \Int (n) | = 2^{n-1}$ and that $NC(n)$ is counted by the 
$n$-th Catalan number:
\[
| NC(n) | = \Cat_n := \frac{ (2n)! }{ n! (n+1)!}, \ \ \forall 
\, n \in \bN .
\]
For a more detailed introduction to the $NC(n)$'s, one can 
for instance consult Lectures 9 and 10 of \cite{NiSp2006}. 
\end{remark}

\vspace{10pt}

Given a non-crossing partition $\pi \in NC(n)$, it is convenient 
to formalize the notion of ``relative nesting'' for blocks of 
$\pi$, as follows.

\begin{notation-and-remark}    \label{def:22}
Let $n$ be in $\bN$ and let $\pi$ be a partition in $NC(n)$. 

\vspace{6pt}

(1) Let $V, W$ be blocks of $\pi$.  We will write ``$V \leqnest W$'' 
to mean that we have the inequalities
\[
\min (W) \leq  \min (V) \mbox{ and } \max (W) \geq \max (V).
\]
We will write ``$V \lnest W$'' 
to mean that $V \leqnest W$ and $V \neq W$.  We will occasionally 
also use the notations $W \geqnest V$ instead of $V \leqnest W$
and $W \gnest V$ instead of $V \lnest W$.

\vspace{6pt}

(2) It is immediate that ``$\leqnest$'' is a partial order relation 
on the set of blocks of $\pi$.  
A block $W \in \pi$ which is maximal with respect to 
$\leqnest$ will be said to be an {\em outer block}.  A block $V \in \pi$ 
which is not outer will be said to be an {\em inner block}.
\end{notation-and-remark}

\begin{remark-and-definition}    \label{rem:23}
Let $n$ be in $\bN$, let $\pi$ be a partition in $NC(n)$, and let $V$ 
be a block of $\pi$.  It is easy to check that the set 
$\{ W \in \pi \mid W \geqnest V \}$ is totally ordered by $\leqnest$.  
That is, we can write 
\begin{equation}   \label{eqn:23a}
\{ W \in \pi \mid W \geqnest V \} = \{ V_1, \ldots , V_k \}
\end{equation}
where $k \geq 1$ and $V_1 \lnest V_2 \lnest \cdots \lnest V_k$.  In 
(\ref{eqn:23a}) we note, in particular, that $V_1 = V$ and that $V_k$ 
is an outer block.  
The {\em depth} of $V$ in $\pi$ is defined as 
\[
\depth_{\pi} (V) := k-1 ,
\]
with $k$ picked from Equation (\ref{eqn:23a}).
If $k \geq 2$ (which is equivalent to saying that 
$\depth_{\pi} (V) \neq 0$, or that $V$ is an inner block), then the 
block $V_2$ appearing in (\ref{eqn:23a}) is called the 
{\em parent-block for $V$}, and will be denoted as $\Parent_{\pi} (V)$.  
The parent-block could be equivalently introduced via the requirement that
\[
\left\{  \begin{array}{cl}
\mbox{(i)}    &  V \lnest \Parent_{\pi} (V) , \mbox{ and }  \\
\mbox{(ii)}   & \mbox{There is no block $V' \in \pi$ such 
                that $V \lnest V' \lnest \Parent_{\pi} (V)$.}
\end{array}   \right.
\]
\end{remark-and-definition}

\begin{remark}    \label{rem:24}
Let $n$ be in $\bN$ and let $\pi$ be a partition in $NC(n)$.

(1) The notion of depth for the blocks of $\pi$ could also
be defined recursively, by postulating that outer blocks 
have depth $0$ and by making the requirement that
\[
\depth_{\pi} (V) = 1 +
\depth_{\pi} \bigl( \, \Parent_{\pi} (V) \, \bigr)
\ \mbox{ for every inner block $V \in \pi$.}
\]

(2) As mentioned in the Introduction, for an
$i \in \{ 1, \ldots , n \}$ we will sometimes write 
``$\depth_{\pi} (i)$'' in order to refer to the depth of 
the block of $\pi$ which contains the number $i$.  Thus 
$\depth_{\pi}$ can be viewed as a special example of 
colouring of $\pi$ (a function from $\{ 1, \ldots , n \}$ to
$\bZ$ which is constant along the blocks of $\pi$).
\end{remark}

\begin{remark}   \label{rem:25}
Let $n$ be in $\bN$ and let $\pi$ be a partition in $NC(n)$.
It is easy to see that one can always list the set of outer 
blocks of $\pi$ as $\{ W_1, \ldots , W_{\ell} \}$ in such 
a way that 
\begin{equation}   \label{eqn:25a}
\left\{  \begin{array}{l}
\min (W_1) = 1, \ \ \max (W_{\ell}) = n, \ \ \mbox{ and } \\
\min (W_{i+1} ) = 1 + \max (W_i) \ \ 
\mbox{ for every $1 \leq i < \ell$.}
\end{array}  \right.
\end{equation}
In the case when the number $\ell$ of outer blocks of $\pi$ is 
$\ell = 1$, the second condition in (\ref{eqn:25a}) is vacuous 
($\pi$ has a unique outer block $W$, with $1,n \in W$).

In the notation from (\ref{eqn:25a}): the interval partition 
\[
\overline{\pi} := \{ J_1, \ldots , J_{\ell} \}
\mbox{ with $J_i := [ \min (W_i) , \max (W_i) ] \cap \bN$, 
for $1 \leq i \leq \ell$}
\]
is sometimes called the {\em closure} of $\pi$ in the set of 
interval-partitions.
Note that knowing what is $\overline{\pi}$ provides exactly the 
same information as knowing the set $\OuterMax ( \pi )$ which 
was introduced in Definition \ref{def:151} of the Introduction.
\end{remark}

\vspace{10pt}

\subsection{Review of free and of Boolean cumulant functionals.}

$\ $

\noindent
Let $( \cA , \varphi )$ be a noncommutative probability space 
(in purely algebraic sense) -- that is, $\cA$ is a unital algebra 
over $\bC$ and $\varphi : \cA \to \bC$ a linear functional with 
$\varphi ( 1_{{ }_{\cA}} ) = 1$.  
In this subsection we briefly review the definition of 
the free and the Boolean cumulants of $( \cA , \varphi )$.
Before starting, we record a customary notation which will 
appear in the formulas for both types of cumulants:
given an $n \in \bN$, a tuple
$( a_1, \ldots , a_n ) \in \cA^n$, and a non-empty subset
$S = \{ i_1, \ldots , i_m \} \subseteq \{ 1, \ldots , n \}$
with $i_1 < \cdots < i_m$, we denote
\begin{equation}  \label{eqn:26a}
( a_1, \ldots , a_n ) \mid S 
:= ( a_{i_1}, \ldots , a_{i_m} ) \in \cA^m.
\end{equation}

$\ $

\begin{definition}   \label{def:26}
Notations as above.

\vspace{6pt}

(1) The {\em free cumulants} associated to $( \cA , \varphi )$ 
are the family of multilinear functionals 
$( \kappa_n : \cA^n \to \bC )_{n=1}^{\infty}$ which 
is uniquely determined by the requirement that
\begin{equation}   \label{eqn:26b}
\varphi (a_1 \cdots a_n)  = 
\sum_{\pi \in NC(n)} \prod_{V \in \pi} 
\kappa_{|V|} ( \, (a_1, \ldots , a_n) \mid V \, ), 
\end{equation}
holding for all $n \in \bN$ and $a_1, \ldots , a_n \in \cA$.

\vspace{6pt}

(2) The {\em Boolean cumulants} associated to $( \cA , \varphi )$ 
are the family of multilinear functionals 
$( \beta_n : \cA^n \to \bC )_{n=1}^{\infty}$ which 
is uniquely determined by the requirement that
\begin{equation}   \label{eqn:26c}
\varphi (a_1 \cdots a_n)  = 
\sum_{\pi \in \Int (n)} \prod_{V \in \pi} 
\beta_{|V|} ( \, (a_1, \ldots , a_n) \mid V \, ), 
\end{equation}
holding for all $n \in \bN$ and $a_1, \ldots , a_n \in \cA$.
\end{definition}

$\ $

\begin{remark}   \label{rem:27}
It is easy to see that the families of equations indicated
in either (\ref{eqn:26b}) or (\ref{eqn:26c}) have unique 
solutions.  Indeed, all that actually matters is that the 
index sets $\Int (n)$ and $NC(n)$ for the summations on the 
right-hand sides of these equations contain the partition,
usually denoted as ``$1_n$'', of the set $\{ 1, \ldots , n \}$
into only one block.  For instance in connection to 
(\ref{eqn:26b}): by separating the term indexed by $1_n$ on 
the right-hand side, this equation can be written as
\begin{equation}   \label{eqn:27a}
\kappa_n (a_1, \ldots , a_n) = 
\varphi (a_1, \cdots , a_n) -
\sum_{ \begin{array}{c}
{\scriptstyle \pi \in NC(n),}   \\
{\scriptstyle \pi \neq 1_n} 
\end{array} } \ \prod_{V \in \pi} 
\kappa_{|V|} ( \, (a_1, \ldots , a_n) \mid V \, ).
\end{equation}
Then (\ref{eqn:27a}) can be used as an explicit definition 
of the functional $\kappa_n$, under the assumption that explicit
formulas for $\kappa_1, \ldots , \kappa_{n-1}$ have already been 
determined.  A similar recursive argument holds in connection to 
solving the system of equations indicated in (\ref{eqn:26c}).

It is in fact not difficult to write in a really explicit 
way some formulas giving the $\kappa_n$'s and the 
$\beta_n$'s in terms of $\varphi$.  This is not needed in the 
present paper, so we only mention that the way to do it goes 
by using some standard elements of ``M\"obius inversion theory
in a partially ordered set'' (as presented e.g. in Chapter 3
of the monograph \cite{St1997}).  

There also is a nice direct formula which expresses Boolean 
cumulants in terms of free cumulants, as follows.
\end{remark}

\begin{proposition}   \label{prop:28}
Let $( \cA , \varphi )$ be a noncommutative probability space,
and let $( \kappa_n : \cA^n \to \bC )_{n=1}^{\infty}$ and
$( \beta_n : \cA^n \to \bC )_{n=1}^{\infty}$ be the free and
respectively the Boolean cumulants associated to 
$( \cA , \varphi )$.  For every $n \in \bN$ and 
$a_1, \ldots , a_n \in \cA$ one has
\begin{equation}   \label{eqn:28a}
\beta_n (a_1, \ldots , a_n) = 
\sum_{ \begin{array}{c}
{\scriptstyle \pi \in NC(n) \ with}   \\
{\scriptstyle unique \ outer \ block} 
\end{array} } \ \prod_{V \in \pi} 
\kappa_{|V|} ( \, (a_1, \ldots , a_n) \mid V \, ).
\end{equation}
\end{proposition}

The proof of Proposition \ref{prop:28} can e.g. be 
obtained by an immediate re-phrasing of the argument 
proving Proposition 3.9 in \cite{BeNi2008}.

\begin{notation-and-remark}   \label{def:27b}
Let $( \cA , \varphi )$ be a noncommutative probability space,
and let $( \kappa_n : \cA^n \to \bC )_{n=1}^{\infty}$ and
$( \beta_n : \cA^n \to \bC )_{n=1}^{\infty}$ be the free and
respectively the Boolean cumulants associated to $( \cA , \varphi )$.
Let $a \in \cA$ be given.  We will use the abbreviations
\[
\beta_n (a) := \beta_n (a, \ldots , a) \mbox{ and }
\kappa_n (a) := \kappa_n (a, \ldots , a), \ \ n \in \bN .
\]
The power series 
\[
M_a (z) = \sum_{n=1}^{\infty} \varphi (a^n) z^n, 
\ R_a (z) = \sum_{n=1}^{\infty} \kappa_n (a) z^n
\mbox{ and } 
\eta_a (z) = \sum_{n=1}^{\infty} \beta_n (a) z^n 
\]
are called the {\em moment series}, the {\em R-transform}
and respectively the {\em $\eta$-series} associated to $a$.

In this paper, an important role is played by $\eta$-series.
We note that, as an easy consequence of the formula (\ref{eqn:26c}) 
connecting moments to Boolean cumulants, one has a very simple 
relation between $\eta_a$ and $M_a$:
\[
M_a (z) = \eta_a (z) / ( 1 - \eta_a (z) ),
\mbox{ or equivalently, }
\eta_a (z) = M_a (z) / ( 1 + M_a (z) ).
\]
\end{notation-and-remark}

\vspace{10pt}

\subsection{Cumulants with products as arguments.}

$\ $

\noindent
When working with cumulants of any kind, it is good to have an
efficient formula for what happens when every argument of the 
cumulant is a product of elements of the underlying algebra.
For free cumulants, this formula was put into evidence in 
\cite{KrSp2000}.  We will need here the analogous fact for 
Boolean cumulants.  In order to state this fact and to explain 
the analogy with \cite{KrSp2000}, we need to use the lattice 
structure (with respect to the partial order by reverse refinement) 
on $NC(n)$ and on $\Int (n)$, so we first do a brief review of
this structure. 

$\ $

\begin{definition}   \label{def:29}
Let $n$ be a positive integer.  

\vspace{6pt}

(1) On $NC(n)$ we consider the partial order by 
{\em reverse refinement}, where for $\pi, \rho \in NC(n)$ we put
\begin{equation}  \label{eqn:29a}
( \pi \leq \rho )  \ \ecdef \ \Bigl( \mbox{
every block of $\rho$ is a union of blocks of $\pi$} \Bigr).  
\end{equation}
The partially ordered set $( NC(n), \leq )$ turns out to be a 
lattice.  That is, every $\pi_1 , \pi_2 \in NC(n)$ have a least 
common upper bound, denoted as $\pi_1 \vee \pi_2$,
and have a greatest common lower bound, denoted as 
$\pi_1 \wedge \pi_2$.  One refers to $\pi_1 \vee \pi_2$ and 
to $\pi_1 \wedge \pi_2$ as the {\em join} and respectively as the 
{\em meet} of $\pi_1$ and $\pi_2$ in $NC(n)$.

We will use the notation $0_n$ for the partition of 
$\{ 1, \ldots , n \}$ into $n$ singleton blocks and the notation 
$1_n$ for the partition of $\{ 1, \ldots , n \}$ into one block.
It is immediate that $0_n, 1_n \in NC(n)$ and that 
$0_n \leq \pi \leq 1_n$ for all $\pi \in NC(n)$.

\vspace{6pt}

(2) Consider the restriction of the partial order by reverse 
refinement from $NC(n)$ to $\Int (n)$.  
For $\pi_1, \pi_2 \in \Int (n)$, the partitions 
$\pi_1 \vee \pi_2, \pi_1 \wedge \pi_2 \in NC(n)$ 
which were defined in (1) above turn out to still belong to 
$\Int (n)$.  As a consequence, $( \Int (n) , \leq )$ is a 
lattice as well, and for $\pi_1 , \pi_2 \in \Int (n)$ there is 
no ambiguity in the meaning of what are
$\pi_1 \vee \pi_2$ and $\pi_1 \wedge \pi_2$ (considering the  
join and meet of $\pi_1$ and $\pi_2$ in $\Int (n)$ gives the 
same result as when considering them in $NC(n)$).

The special partitions $0_n$ and $1_n$ considered in (1) 
belong to $\Int (n)$, hence they also serve as minimum and
maximum elements for the poset $( \Int (n) , \leq )$.
\end{definition}

$\ $

The formula for Boolean cumulants with products as entries is 
then stated as follows.

\begin{proposition} \label{prop:210}
Let $( \cA , \varphi )$ be a noncommutative probability space,
and let $( \beta_n : \cA^n \to \bC )_{n=1}^{\infty}$ be the 
family of Boolean cumulant functionals of $( \cA , \varphi )$.  
Consider a Boolean cumulant of the form 
$\beta_m ( x_1, \ldots , x_m )$ where each of the elements 
$x_1, \ldots , x_m \in \cA$ is written as a product:
\[
x_1 = a_1 \cdots a_{i(1)}, 
x_2 = a_{i(1)+1} \cdots a_{i(2)}, \ldots , 
x_m = a_{i(m-1)+1} \cdots a_{i(m)}, 
\]
where $1 \leq i(1) < i(2) < \cdots < i(m) =:n$ are some positive 
integers, and where $a_1, \ldots , a_n \in \cA$.  Then one has
\begin{equation}   \label{eqn:210a}
\beta_m ( x_1, \ldots , x_m ) = 
\sum_{  \begin{array}{c}
{\scriptstyle \pi \in \Int (n) \ such} \\
{\scriptstyle that \ \pi \vee \sigma = 1_n} 
\end{array} \ }
\ \prod_{V \in \pi} \beta_{|V|} ( ( a_1, \ldots , a_n) \mid V),
\end{equation}
with 
\begin{equation}   \label{eqn:210b}
\sigma := \{ \, \{ 1, \ldots , i(1) \} , 
\, \{ i(1) + 1, \ldots , i(2) \} , \ldots ,
\{ i(m-1) + 1, \ldots , i(m) \} \, \} \in \Int (n). 
\end{equation}
\end{proposition} 

\vspace{6pt}

The proof of Proposition \ref{prop:210} is left as an exercise 
to the reader.  A way to do it is by going over the development 
presented on pages 178-180 of \cite{NiSp2006} about free
cumulants with products as arguments, and by replacing everywhere 
on those pages the occurrences of lattices of non-crossing 
partitions by occurrences of lattices of interval partitions.  
The statement of Proposition \ref{prop:210} should then emerge 
as the Boolean analogue of Theorem 11.12(2) of \cite{NiSp2006}.

\vspace{6pt}

\begin{remark}   \label{rem:211}
The lattice $\Int (n)$ is in fact a Boolean lattice.  Indeed, one has 
a natural bijection which identifies $\Int (n)$ to the lattice of 
subsets of $\{ 1, \ldots , n-1 \}$, by sending a partition 
$\pi = \{ J_1, \ldots , J_k \} \in \Int (n )$ to the set
$\{ \max (J_1), \ldots , \max (J_k) \} \setminus \{ n \} 
\subseteq \{ 1, \ldots , n-1 \}$.
By using this fact it is easy to see that, with $\sigma$ as defined 
in Equation (\ref{eqn:210b}), a partition $\pi \in \Int (n)$ has 
\[
\bigl( \pi \vee \sigma = 1_n \bigr) \ \Leftrightarrow 
\Bigl(  \begin{array}{c}
\mbox{$i(p)$ and $i(p)+1$ belong to the same}  \\
\mbox{block of $\pi$, for all $1 \leq p \leq m-1$}
\end{array} \Bigr) .
\]
This allows for a somewhat more convenient re-phrasing of the 
join condition invoked on the right-hand side of
Equation (\ref{eqn:210a}). 
\end{remark}

$\ $

\section{The partial order $\ll$, VNRP, and the proof of 
Theorem \ref{thm:131}}

\subsection{The partial order \boldmath{$\ll$} on $NC(n)$, 
and its upper ideals.}

$\ $

\noindent
In this paper we also make use of another partial order
relation on $NC(n)$, coarser than reverse refinement, 
which is denoted as ``$\ll$''.  The 
partial order $\ll$ has been used for some time in free 
probability (starting with \cite{BeNi2008}), in the description 
of relations between free and Boolean cumulants.  

$\ $

\begin{defrem}       \label{def:31}
{\em (The partial order ``$\ll$''.)}

(1) For $\pi , \rho \in NC(n)$, we will write $\pi \ll \rho$ to mean 
that $\pi \leq \rho$ and that, in addition, for every block $W$ of 
$\rho$ there exists a block $V$ of $\pi$ such that $\min (W), \max (W) \in V$.  

(2) Since in this paper we have a lot of occurrences of the special case
``$\pi \ll 1_n$'', let us record the obvious fact that this simply amounts 
to requiring $\pi$ to have a unique outer block $W$, with $1,n \in W$.

(3) It is immediate that $\Int (n)$ is precisely equal to the 
set of maximal elements of the poset $( NC(n), \ll )$.

(4) A significant point about the partial order $\ll$ on $NC(n)$ 
is that we have a nice structure for its upper ideals, that is, 
for the sets of non-crossing partitions of the form
\begin{equation}   \label{eqn:31a}
\{ \rho \in NC(n) \mid \rho \gg \pi \} , 
\ \mbox{ for a fixed $\pi \in NC(n)$.}
\end{equation}
This was noticed in Section 2 of \cite{BeNi2008}, but the discussion 
around the set (\ref{eqn:31a}) was mostly done in a proof (cf. proof 
of Proposition 2.13 in \cite{BeNi2008}), and it will be useful for 
our present purposes to spell that out in more detail.  It turns out 
to be convenient to use a notion of ``projection map'' for blocks of 
a fixed $\pi \in NC(n)$, as introduced in the next definition.
\end{defrem}

$\ $

\begin{definition}   \label{def:32}
Let $n$ be in $\bN$ and let $\pi$ be a partition in $NC(n)$. 
A {\em block-projection} for $\pi$ is a map
$\Phi : \pi \to \pi$ which has the following properties.

(i) $\Phi$ is a projection map; that is, $\Phi \circ \Phi = \Phi$.

(ii) If $A,B \in \pi$ and if $A \leqnest B$, then it 
follows that $\Phi (A) \leqnest \Phi (B)$.

(iii) $A \leqnest \Phi (A)$ for all $A \in \pi$.

\vspace{6pt}

\noindent
For such a $\Phi$ we will denote
$\Ran ( \Phi ) := \{ B \in \pi \mid \exists \, A \in \pi 
\mbox{ such that } \Phi (A) = B \}$.
Note that, due to the property (i) satisfied by $\Phi$, 
we can also
write $\Ran ( \Phi ) = \{ B \in \pi \mid \Phi (B) = B \}$.
\end{definition}

\begin{remark}    \label{rem:33}
Let $n$ be in $\bN$ and let $\pi$ be a partition in $NC(n)$. 

(1) Let $\Phi : \pi \to \pi$ be a block-projection for $\pi$.
Property (iii) satisfied by $\Phi$ implies that 
$\Ran ( \Phi )$ contains all the outer blocks of $\pi$.

(2) Let $\Phi , \Psi : \pi \to \pi$ be block-projections for $\pi$,
and suppose that $\Ran ( \Phi ) = \Ran ( \Psi )$.  Then $\Phi = \Psi$.
Indeed, for every block $A \in \pi$ we can apply $\Psi$ to both sides 
of the relation $A \leqnest \Phi (A)$ to get that 
$\Psi (A) \leqnest \Psi ( \, \Phi (A) \, ) = \Phi (A)$,
(where the latter equality holds because 
$\Phi (A) \in \Ran ( \Phi ) = \Ran ( \Psi )$, hence $\Phi (A)$ is 
fixed by $\Psi$).  A symmetric argument gives that
$\Phi (A) \leqnest \Psi (A)$, and it follows that $\Phi (A) = \Psi (A)$,
as required. 
\end{remark}

\begin{lemma}    \label{lemma:34}
Let $n$ be in $\bN$ and let $\pi$ be a partition in $NC(n)$. 
Let $\fM$ be a subset of $\pi$ such that $\fM$ contains all the 
outer blocks of $\pi$.  Then there exists a block-projection 
$\Phi : \pi \to \pi$, uniquely determined, such that 
$\Ran ( \Phi ) = \fM$.
\end{lemma}

\begin{proof}  Uniqueness of $\Phi$ follows from the preceding remark.
In order to prove existence, we use the following prescription
to define $\Phi$:
\begin{equation}    \label{eqn:34a}
\left\{  \begin{array}{l}
\mbox{if $A \in \fM$, then $\Phi (A) = A$;}              \\
\mbox{if $A \in \pi \setminus \fM$, then 
      $\Phi (A) = \Phi( \, \Parent_{\pi} (A) \, )$.}      \\
\end{array}   \right.
\end{equation}
The definition proposed via Equations (\ref{eqn:34a}) is consistent because 
if we start with any $A \in \pi$ and do iterations of the $\Parent_{\pi}$ map
on it, we will eventually have to find a block that belongs to $\fM$.
Or more precisely: if we start with $A \in \pi$ and we write explicitly
$\{ B \in \pi \mid B \geqnest A \} = \{ B_1, \ldots , B_k \}$ in the way
indicated in Remark \ref{def:22}(4), then Equations (\ref{eqn:34a}) 
define $\Phi (A) = B_j$ with 
$j := \min \Bigl\{ i \in \{ 1, \ldots , k \} \mid B_i \in \fM \Bigr\}$.
\end{proof}

\begin{remark}    \label{rem:35}
One has a natural construction of block-projection map 
$\Phi : \pi \to \pi$ which arises whenever we are given two partitions 
$\pi , \rho \in NC(n)$ such that $\pi \ll \rho$.  Recall that, in this 
situation, for every block $X \in \rho$ there exists a block
$B \in \pi$ such that
\begin{equation}   \label{eqn:35a}
B \subseteq X \mbox{ and } \min (B) = \min (X), \ \max (B) = \max (X).
\end{equation}  
We then define $\Phi : \pi \to \pi$ as follows: for every $A \in \pi$
we consider the (unique) block $X \in \rho$ such that 
$X \supseteq A$, and then we define $\Phi (A) := B$, where $B$ is as in 
(\ref{eqn:35a}).  It is easy to check that the map $\Phi : \pi \to \pi$
defined in this way fulfills the conditions (i), (ii) and (iii) from 
Definition \ref{def:32}, hence is indeed a block-projection map for $\pi$.

Let us record that, in the terminology introduced in 
\cite{BeNi2008}, a block $B$ as in (\ref{eqn:35a}) 
is said to be a {\em $\rho$-special} block of $\pi$.  The 
block-projection $\Phi$ constructed above is characterized by the fact 
that $\Ran ( \Phi )$ is precisely the set of all $\rho$-special blocks 
of $\pi$.
\end{remark}

We now come to the main point concerning the set of partitions indicated
in (\ref{eqn:31a}), namely that it is actually ``parametrized '' by the set 
of block-projection maps for $\pi$, where the parametrization is just the
inverse of the natural construction indicated in Remark \ref{rem:35}.
The formal statement of how this works is recorded in the next proposition.

\begin{proposition}    \label{prop:36}
Let $n$ be in $\bN$, let $\pi$ be a partition in $NC(n)$, and let 
$\Phi : \pi \to \pi$ be a block-projection map.  Then there exists
$\rho \in NC(n)$, uniquely determined, such that $\rho \gg \pi$ and
such that $\Phi$ is obtained from $\pi$ and $\rho$ by using the recipe 
described in Remark \ref{rem:35}.  If we list the range of $\Phi$ as 
$\Ran ( \Phi ) =: \{ B_1, \ldots, B_p \}$, then the partition $\rho$
can be described explicitly as $\rho = \{ X_1, \ldots , X_p \}$, 
where
\begin{equation}    \label{eqn:36a}
X_j = \cup_{A \in \Phi^{-1} (B_j)} \ A, \ \ 1 \leq j \leq p.
\end{equation}
\end{proposition}

The proof of Proposition \ref{prop:36} amounts essentially to reproducing
the proof of Proposition 2.13 from \cite{BeNi2008}, but where we work with 
$\Phi$ itself rather than writing all the arguments in terms of the set 
of blocks $\Ran ( \Phi )$.  We note that in view of Lemma \ref{lemma:34}, 
the parametrization of $\{ \rho \in NC(n) \mid \rho \gg \pi \}$ in terms 
of block-projections for $\pi$ can also be viewed as a parametrization in 
terms of subsets of $\pi$ which contain all the outer blocks -- this is, 
actually, what was observed in Proposition 2.13 of \cite{BeNi2008} and in 
the proof of that proposition.

We conclude the discussion about $\ll$ with an observation that will be 
needed in the next subsection.  This observation does not depend on 
Proposition \ref{prop:36}, it is just a direct consequence of how the 
notion of ``$\rho$-special block of $\pi$'', is defined.  It goes as 
follows.

\begin{lemma}   \label{lemma:37}
Let $n$ be in $\bN$, and let $\pi , \rho \in NC(n)$ be such that 
$\pi \ll \rho$.  Let $A$ be a $\rho$-special block of $\pi$ which is not 
outer, and let $B = \Parent_{\pi} (A) \in \pi$.  Let $X,Y$ be the blocks 
of $\rho$ determined by the requirements that $X \supseteq A$ and 
$Y \supseteq B$.  Then $Y = \Parent_{\rho} (X)$.
\hfill $\square$
\end{lemma}

The proof of Lemma \ref{lemma:37} is done by an elementary argument, 
directly from the definitions of the notions involved.  (One must keep in 
mind, of course, that the hypothesis ``$A$ is $\rho$-special'' means,
by definition, that $\min (A) = \min (X)$ and $\max (A) = \max (X)$.)

\vspace{10pt}

\subsection{VNRP, and the proof of Theorem \ref{thm:131}.}

$\ $

\noindent
Throughout this whole subsection we fix the data used in the statement 
of Theorem \ref{thm:131}.  That is, we fix two positive integers $m$ and 
$s$, and a function $c : \{ 1, \ldots , m \} \to \{ 1, \ldots , s \}$.
(We think of $c$ as of a ``colouring of $\{ 1, \ldots , m \}$ in $s$
colours''.)  
We will denote by $NC(m;c)$ the subset of $NC(m)$ defined by
\[
NC(m;c) := \{ \sigma \in NC(m) \mid
\mbox{ $c$ is constant on every block of $\sigma$} \} .
\]
For $\sigma \in NC(m;c)$ and $A \in \sigma$, we will use the notation 
$c(A)$ for the common value $c(a) \in \{ 1, \ldots , s \}$ taken by 
$c$ on all $a \in A$.
Note that on $NC(m;c)$ we have two partial order relations ``$\leq$'' 
(reverse refinement) and ``$\ll$'', induced from $NC(n)$.

The definition of VNRP goes as follows.

\begin{definition}   \label{def:41}
A partition $\sigma \in NC(m;c)$ will be said to have the 
{\em vertical no-repeat property} with respect to $c$ when the 
following happens: for every inner block $A \in \sigma$, 
one has
\[
c( \, \Parent_{\sigma} (A) \, ) \neq c(A).
\]
As already done in the Introduction, we will refer to the 
vertical no-repeat property by using the acronym ``VNRP''.
(Note that if the number of colours $s$ would happen to be 
$s=2$, VNRP could also go under the name of 
``vertical alternance property''.)
\end{definition}

$\ $

Our goal for the section is to prove Theorem \ref{thm:131} stated in 
the Introduction.  In order to do so, we start with an adjustment
of Proposition \ref{prop:36} to the present framework which uses 
coloured partitions from $NC(m;c)$.

\begin{proposition}  \label{prop:Y3}
Let $\sigma$ be a partition in $NC(m;c)$.  
Let $\Phi : \sigma \to \sigma$ be a block-projection map, and let 
$\rho \in NC(m)$ be the partition with $\rho \gg \sigma$ which is 
parametrized by $\Phi$ in Proposition \ref{prop:36}.  We have that:
\begin{equation}   \label{eqn:Y3a}
\Bigl( \rho \in NC(m;c) \Bigr) \ \Leftrightarrow
\ \Bigl( c( \Phi (A) ) = c(A), \ \forall \, A \in \sigma \Bigr) .
\end{equation} 
\end{proposition}

\begin{proof} ``$\Rightarrow$'' From the concrete description of $\rho$ 
provided by Equation (\ref{eqn:36a}) of Proposition \ref{prop:36}, it is 
clear that for every $A \in \sigma$, the blocks $A$ and $\Phi (A)$ are 
contained in the same block $X$ of $\rho$.  The latter fact implies in 
particular that $c(A) = c(X) = c( \Phi (A) )$.  The condition on $\Phi$ 
listed on the right-hand side of (\ref{eqn:Y3a}) thus follows.

\vspace{6pt}

``$\Leftarrow$''  Let $X$ be a block of $\rho$, and consider the block
$B$ of $\sigma$ such that $\min (X), \max (X) \in B$. The explicit description 
of $\rho$ provided by Equation (\ref{eqn:36a}) in Proposition \ref{prop:36}
tells us that
\begin{equation}   \label{eqn:Y3b}
X = \cup_{A \in \Phi^{-1} (B)} \ A.
\end{equation}
Denoting $c(B) = s_o \in \{ 1, \ldots , s \}$, we see that $c(A) = s_o$ for 
every $A$ appearing in the union from (\ref{eqn:Y3b}) -- indeed, one has 
$\Phi (A) = B$, hence the hypothesis $c(A) = c( \Phi (A) )$ gives 
$c(A) = c(B) = s_o$.  This makes it clear that $c$ is constantly equal to 
$s_o$ on $X$, and completes the verification that $\rho \in NC(n;c)$.
\end{proof}

\begin{corollary}  \label{cor:Y4}
Let $\sigma, \rho$ be partitions in $NC(m;c)$ such that 
$\sigma \ll \rho$.  Let $\Phi : \sigma \to \sigma$ be the block-projection 
map which corresponds to $\rho$ in Remark \ref{rem:35}.  One has
\begin{equation}   \label{eqn:Y4a}
\Ran ( \Phi ) \supseteq
\{ A \in \sigma \mid A \mbox{ is inner and } 
   c( \Parent_{\sigma} (A) ) \neq c(A) \} .
\end{equation}
\end{corollary}

\begin{proof}  We prove the reverse inclusion for the complements of the 
sets indicated in in (\ref{eqn:Y4a}).  That is: we pick 
$A_o \in \sigma \setminus \Ran ( \Phi )$, 
and we prove that $A_o$ belongs to the complement
of the set on the right-hand side of (\ref{eqn:Y4a}).  The condition 
$A_o \not\in \Ran ( \Phi )$ implies in particular that $A_o$ is inner,
so what we have to prove is the equality 
$c( \Parent_{\sigma} (A_o) ) = c(A_o)$.

We denote $\Parent_{\sigma} (A_o) = A_1$. It is easy to see (directly 
from the properties of $\Phi$ listed in Definition \ref{def:32}) 
that the assumption $A_o \not\in \Ran ( \Phi )$ (which is equivalent to
$\Phi ( A_o ) \neq A_o$, hence to $A_o \lnest \Phi (A_o)$) entails the
equality $\Phi (A_o) = \Phi (A_1)$.  On the other hand, the assumption 
that $\rho \in NC(m;c)$ entails, via Proposition \ref{prop:Y3}, the 
equalities $c(A_o) = c( \Phi (A_o) )$ and $c(A_1) = c( \Phi (A_1) )$.
By putting all these things together we find that 
$c(A_o) = c( \Phi (A_o) )$
= $c( \Phi (A_1) ) = c(A_1)$.
Hence $c( \Parent_{\sigma} (A_o) ) = c(A_o)$, as required.
\end{proof}

$\ $

The next proposition addresses the uniqueness part in the statement of 
Theorem \ref{thm:131}, by giving an explicit description of the set of
$\tau$-special blocks of $\sigma$, for the partition 
$\tau \gg \sigma $ which is needed in the conclusion of the theorem.

\begin{proposition}   \label{prop:Y5}
Let $\sigma$ be a partition in $NC(m;c)$.  Suppose that
$\tau \in NC(m;c)$ has VNRP, and is such that $\tau \gg \sigma$.  
Then the set of $\tau$-special blocks of $\sigma$ is equal to
\begin{equation}   \label{eqn:Y5a}
\{ A \in \sigma \mid A \mbox{ is outer} \} \cup
\{ A \in \sigma \mid A \mbox{ is inner and } 
   c( \Parent_{\pi} (A) ) \neq c(A) \} .
\end{equation}
\end{proposition}

\begin{proof}  Let $\Phi : \sigma \to \sigma$ be the block-projection 
map which parametrizes $\tau$ in the way described in Proposition 
\ref{prop:36}.  The set of $\tau$-special blocks of $\sigma$ is thus
the same as $\Ran ( \Phi )$.  Corollary \ref{cor:Y4} then assures us 
that the set of $\tau$-special blocks of $\sigma$ contains all blocks 
$A \in \sigma$ such that $A$ is inner and 
$c ( \Parent_{\pi} (A) ) \neq c(A)$.
Since $\Ran ( \Phi )$ is also sure to contain all the outer blocks 
of $\sigma$ (Remark \ref{rem:33}(1)), it follows that $\Ran ( \Phi )$ 
must contain the set of blocks indicated in formula (\ref{eqn:Y5a}).  

Let us assume, for contradiction, that $\Ran ( \Phi )$ is strictly larger
than the set from (\ref{eqn:Y5a}), i.e. that it contains a block 
$V \in \sigma$ which is inner and has 
$c( \Parent_{\sigma} (V) ) = c(V)$. 
This $V$ is a $\tau$-special block of $\sigma$ (since it is in 
$\Ran ( \Phi )$), hence we can use Lemma \ref{lemma:37} with $A := V$ 
and $B := \Parent_{\sigma} (V)$.  Denoting by $X,Y$ the blocks of $\tau$
which contain $A$ and $B$, respectively, we get from Lemma \ref{lemma:37}
that $\Parent_{\tau} (X) = Y$.  Now,  we must have $c(A) = c(X)$ and 
$c(B) = c(Y)$ (since $A \subseteq X$ and $B \subseteq Y$), so from 
$c(A) = c(B)$ it follows that $c(X) = c(Y)$. This contradicts the 
VNRP of $\tau$, and concludes the proof.
\end{proof}

For the existence part in Theorem \ref{thm:131}, one could go 
by showing directly that the ``candidate for $\tau$'' suggested by 
Proposition \ref{prop:Y5} does indeed the required job.  We leave 
this approach as an exercise to the interested reader, and we just
invoke here a simple maximality argument.

\begin{lemma}    \label{lem:46}
Consider the partial order given by $\ll$ on $NC(m;c)$.
Every maximal element of $( NC(m;c) , \ll )$ has the 
VNRP property.
\end{lemma}

\begin{proof}  We prove the contrapositive: if 
$\pi \in NC(m;c)$ does not have VNRP, then it cannot be a 
maximal element with respect to $\ll$.  Indeed, 
let us pick a $\pi \in NC(m;c)$ without VNRP, and let 
$V, V'$ be blocks of $\pi$ such that 
$V' = \mbox{Parent}_{\pi} (V)$ but $V, V'$ have the same colour.  
Let $\rho$ be the partition of $\{1, \ldots , n \}$ which is 
obtained out of $\pi$ by joining together the blocks $V$ and $V'$.  
Then $\rho \in NC(m;c)$ and $\pi \ll \rho$ (this is a slight 
modification of Lemma 6.4.3 from \cite{BeNi2008}), showing in 
particular that $\pi$ is not maximal with respect to $\ll$.
\end{proof}

$\ $

\begin{remark}   \label{rem:46}
We leave it as an exercise to the reader to check that the 
converse of Lemma \ref{lem:46} is true as well: the maximal 
elements of $( NC(m;c), \ll )$ are precisely the partitions 
which have VNRP with respect to the colouring $c$.
\end{remark}

$\ $

\begin{ad-hoc}
{\bf Proof of Theorem \ref{thm:131}.}  
We fix a partition $\sigma \in NC(m;c)$ and we have to prove that 
there exists a $\tau \in NC(m;c)$, uniquely determined, such that 
$\sigma \ll \tau$ and such that $\tau$ has VNRP with respect to $c$.
And indeed: the uniqueness of $\tau$ with the required properties 
follows from Proposition \ref{prop:Y5}.  On the other hand, the 
existence of $\tau$ follows from Lemma \ref{lem:46} and the fact 
that $\sigma$ must have a majorant which is maximal with respect 
to $\ll$.
\hfill $\square$
\end{ad-hoc}

$\ $

\section{Free independence in terms of Boolean cumulants, via VNRP}
\label{sec:Lemmas}

Based on Theorem \ref{thm:131}, one gets the characterization of 
free independence in terms of Boolean cumulants which was announced 
in Theorem \ref{thm:141} from the Introduction.

$\ $

\begin{ad-hoc}
{\bf Proof of Theorem \ref{thm:141}.}
Recall that in this theorem we are given a noncommutative probability 
space $( \cA , \varphi )$ and some unital subalgebras 
$\cA_1,\ldots,\cA_s\subseteq \cA$, and we have to prove the 
equivalence of two statements $(1)$ and $(2)$ concerning  
$\cA_1, \ldots, \cA_s$.
Throughout the proof we use the notation 
$( \beta_n : \cA^n \to \bC )_{n=1}^{\infty}$ and respectively 
$( \kappa_n : \cA^n \to \bC )_{n=1}^{\infty}$ for the families 
of Boolean and respectively free cumulant functionals of 
$( \cA , \varphi )$.

\vspace{6pt}

{\em Proof that $(1) \Rightarrow (2)$.}
Here we know that 
$\cA_1,\ldots,\cA_s$ are freely independent with respect to 
$\varphi$.  We consider an $n \in \bN$, a colouring 
$c : \{ 1, \ldots , n \} \to \{ 1, \ldots , s \}$, and some 
elements $a_1 \in \cA_{c(1)}, \ldots a_n \in \cA_{c(n)}$, and 
we have to prove that 
\begin{equation}    \label{eqn:51a}
\beta_n (a_1, \ldots , a_n) = 
\sum_{  \begin{array}{c}
{\scriptstyle \pi \in NC(n;c),}  \
{\scriptstyle \pi \ll 1_n ,} \\
{\scriptstyle \pi \ has \ VNRP}
\end{array}  } 
\ \prod_{V \in \pi} \beta_{|V|} ( (a_1, \ldots , a_n) \mid V ).
\end{equation}
To that end, we write
\begin{align*}
\beta_n (a_1, \ldots , a_n)
& = \sum_{\pi \in NC(n), \ \pi \ll 1_n}
\ \prod_{V \in \pi} \kappa_{|V|} ( (a_1, \ldots , a_n) \mid V)   \\
& = \sum_{ \pi \in NC(n,c), \ \pi \ll 1_n} 
\ \prod_{V \in \pi} \kappa_{|V|} ( (a_1, \ldots , a_n) \mid V),
\end{align*}
where at the first equality sign we used the formula expressing 
Boolean cumulants in terms of free cumulants, and at the second
equality sign we used the vanishing of mixed free cumulants with 
entries from the free subalgebras $\cA_1, \ldots, \cA_s$.

We now invoke Theorem \ref{thm:131}, and group the partitions 
``$\pi \in NC(n,c), \ \pi \ll 1_n$'' which index the latter sum
according to the unique $\rho \in NC(n;c)$ such that $\pi \ll \rho$
and $\rho$ has VNRP with respect to $c$.  When doing so, we continue 
the above equalities with
\[
= \sum_{  \begin{array}{c}
{\scriptstyle \rho \in NC(n,c),\ \rho \ll 1_n ,}  \\
{\scriptstyle \rho \ has \ VNRP}
\end{array}  } \, \Bigl( \ \sum_{\pi \ll \rho}
\prod_{V \in \pi} \kappa_{|V|} 
((a_1, \ldots , a_n) \mid V) \, \Bigr).
\]
Finally, in the latter double sum we note that the inside summation
comes to 

\noindent
$\prod_{V \in \rho} \beta_{|V|} ((a_1, \ldots , a_n) \mid V)$ (again 
due to how Boolean cumulants are expressed in terms of free cumulants).  
This leads precisely to the formula for $\beta_n (a_1, \ldots , a_n)$ 
that was stated in Equation (\ref{eqn:51a}).

\vspace{6pt}

\noindent
{\em Proof that $(2) \Rightarrow (1)$.}
In order to prove that $\cA_1,\ldots,\cA_s$ are free, we consider
the free product $\left(\widetilde{\cA},\widetilde{\varphi}\right)$
= $\ast\left(\cA_i,\varphi|_{\cA_i}\right)_{i=1,\ldots,s}$. Then 
for any $a_1,\ldots,a_n$ such that 
$a_k\in\cA_{i(k)}$, formula \eqref{eqn:51a} holds for Boolean 
cumulants related with $\varphi$, denoted by $\beta_{n}^{\varphi}$ 
by assumption. On the other hand \eqref{eqn:51a} holds also for 
$\beta_{n}^{\widetilde{\varphi}}$, i.e. Boolean cumulants related 
with $\widetilde{\varphi}$, since $a_1,\ldots,a_n$ are free wrt 
$\widetilde{\varphi}$ hence the implication $(1) \Rightarrow (2)$, 
proved above, can be applied.  Since for any $a_1,\ldots,a_n$ 
such that $a_k\in\cA_{i(k)}$ we have 
$\beta_{n}^{\varphi}(a_1,\ldots,a_n)=
\beta_{n}^{\widetilde{\varphi}}(a_1,\ldots,a_n)$, 
then also all joint moments with respect to $\varphi$ and 
$\widetilde{\varphi}$ coincide. Since $\cA_1,\ldots,\cA_s$ are free 
with respect to $\widetilde{\varphi}$, we get that $\cA_1,\ldots,\cA_s$ 
are free with respect to $\varphi$.
\hfill $\square$
\end{ad-hoc}

\begin{example}   \label{ex:42}
For the sake of clarity, we give a concrete example of how the 
formula (\ref{eqn:51a}) works.  Let $\cA_1, \cA_2$ be freely 
independent subalgebras of $\cA$. Suppose we pick elements 
$a, a', a'' \in \cA_1$ and $b, b', b'' \in \cA_2$, and 
we are interested in the Boolean cumulant
$\beta_6 ( a,b, a', b', b'', a'')$.

\begin{minipage}{1\linewidth}
\begin{center}
\begin{tikzpicture}
			\draw[red,thick] (-6,0) -- (-6,0.25);
			\draw[blue,thick] (-4,0) -- (-4,0.25);
			\draw[red,thick] (-8,0) to[out=90,in=90] node [sloped,above] {} (-3,0);
			\draw[blue,thick] (-7,0) to[out=90,in=90] node [sloped,above] {} (-5,0);
				
			\node[below,red] (1) at (-8,0) {$a$};
			\node[below,blue] (2) at (-7,0) {$b$};
			\node[below,red] (3) at (-6,0) {$a'$};
			\node[below,blue] (4) at (-5,0) {$b'$};
			\node[below,blue] (5) at (-4,0) {$b''$};
			\node[below,red] (6) at (-3,0) {$a''$};

			\draw[red,thick] (2,0) -- (2,0.25);
			\draw[red,thick] (0,0) to[out=90,in=90] node [sloped,above] {} (5,0);
			\draw[blue,thick] (1,0) to[out=90,in=90] node [sloped,above] {} (3,0);
			\draw[blue,thick] (3,0) to[out=90,in=90] node [sloped,above] {} (4,0);

			\node[below,red] (7) at (0,0) {$a$};
			\node[below,blue] (8) at (1,0) {$b$};
			\node[below,red] (9) at (2,0) {$a'$};
			\node[below,blue] (10) at (3,0) {$b'$};
			\node[below,blue] (11) at (4,0) {$b''$};
			\node[below,red] (12) at (5,0) {$a''$};
\end{tikzpicture}
\end{center}
\end{minipage}

\begin{minipage}{1\linewidth}
\begin{center}
\begin{tikzpicture}
		\draw[blue,thick] (-4,0) -- (-4,0.25);
		\draw[blue,thick] (-7,0) -- (-7,0.25);
		\draw[blue,thick] (-5,0) -- (-5,0.25);
		\draw[red,thick] (-8,0) to[out=90,in=90] node [sloped,above] {} (-6,0);
		\draw[red,thick] (-6,0) to[out=90,in=90] node [sloped,above] {} (-3,0);
		
		\node[below,red] (1) at (-8,0) {$a$};
		\node[below,blue] (2) at (-7,0) {$b$};
		\node[below,red] (3) at (-6,0) {$a'$};
		\node[below,blue] (4) at (-5,0) {$b'$};
		\node[below,blue] (5) at (-4,0) {$b''$};
		\node[below,red] (6) at (-3,0) {$a''$};

		\draw[blue,thick] (1,0) -- (1,0.25);
		\draw[red,thick] (0,0) to[out=90,in=90] node [sloped,above] {} (2,0);
		\draw[red,thick] (2,0) to[out=90,in=90] node [sloped,above] {} (5,0);
		\draw[blue,thick] (3,0) to[out=90,in=90] node [sloped,above] {} (4,0);

		\node[below,red] (7) at (0,0) {$a$};
		\node[below,blue] (8) at (1,0) {$b$};
		\node[below,red] (9) at (2,0) {$a'$};
		\node[below,blue] (10) at (3,0) {$b'$};
		\node[below,blue] (11) at (4,0) {$b''$};
		\node[below,red] (12) at (5,0) {$a''$};
\end{tikzpicture}
\end{center}
\end{minipage}

\begin{center}
{\bf Figure 3(a).}
{\em Some coloured partitions in $NC(6)$, with VNRP.}
\end{center}

$\ $

\noindent
We get
\begin{equation}   \label{eqn:42a}
\beta_6 (a,b,a',b',b'',a'') =
\begin{array}[t]{l}
\beta_3 (a,a',a'') \beta_1 (b) \beta_1 (b') \beta_1 (b'') 
+ \beta_3 (a,a',a'') \beta_1 (b) \beta_2 (b', b'')         \\
+ \beta_2 (a,a'') \beta_1 (a') \beta_3 (b,b',b'')         
+ \beta_2 (a,a'') \beta_1 (a') \beta_2 (b,b') \beta_1 (b'') ,
\end{array}
\end{equation}
where the four terms on the right-hand side of the above equation 
correspond to the four partitions in $NC(6)$ that are listed in 
Figure 3(a).  It is instructive to note that one also has 
two partitions in $NC(6)$, shown in Figure 3(b), which satisfy the 
colouring condition (i.e. they separate $a$'s from $b$'s in the 
tuple $(a,b,a',b',b'',a'')$) and also satisfy the requirement of 
having a unique outer block), but don't contribute to the 
sum on the right-hand side of (\ref{eqn:42a}) because they don't
have VNRP.
\end{example}

\begin{minipage}{1\linewidth}
\begin{center}
\begin{tikzpicture}
		\draw[red,thick] (-8,0) to[out=90,in=90] node [sloped,above] {} (-3,0);
		\draw[blue,thick] (-7,0) to[out=90,in=90] node [sloped,above] {} (-4,0);
		\draw[red,thick] (-6,0) -- (-6,0.25);
		\draw[blue,thick] (-5,0) -- (-5,0.25);		

		\node[below,red] (1) at (-8,0) {$a$};
		\node[below,blue] (2) at (-7,0) {$b$};
		\node[below,red] (3) at (-6,0) {$a'$};
		\node[below,blue] (4) at (-5,0) {$b'$};
		\node[below,blue] (5) at (-4,0) {$b''$};
		\node[below,red] (6) at (-3,0) {$a''$};

		\draw[red,thick] (0,0) to[out=90,in=90] node [sloped,above] {} (5,0);
		\draw[blue,thick] (3,0) to[out=90,in=90] node [sloped,above] {} (4,0);
		\draw[blue,thick] (1,0) -- (1,0.25);
		\draw[red,thick] (2,0) -- (2,0.25);		

		\node[below,red] (7) at (0,0) {$a$};
		\node[below,blue] (8) at (1,0) {$b$};
		\node[below,red] (9) at (2,0) {$a'$};
		\node[below,blue] (10) at (3,0) {$b'$};
		\node[below,blue] (11) at (4,0) {$b''$};
		\node[below,red] (12) at (5,0) {$a''$};
\end{tikzpicture}
\end{center}
\end{minipage}

\begin{center}
{\bf Figure 3(b).} 
{\em Some coloured partitions in $NC(6)$, without VNRP.}
\end{center}

$\ $

\begin{remark}   \label{rem:52}
It is useful to note a special situation when we are sure to 
get ``vanishing of mixed Boolean cumulants with free arguments'': 
consider the setting of Theorem \ref{thm:141}, and let
$c : \{ 1, \ldots , n \} \to \{ 1, \ldots , s \}$
be a colouring such that $c(1) \neq c(n)$.  Then for 
every $a_1 \in \cA_{c(1)}, \ldots a_n \in \cA_{c(n)}$, 
one has that $\beta_n (a_1, \ldots , a_n) = 0$.  
Indeed, in this special case the index set for the 
summation on the right-hand side of Equation (\ref{eqn:51a}) 
is the empty set.
\end{remark}

$\ $

In the remaining part of this section, we record some easy 
consequences of Theorem \ref{thm:141}.  First, we note that 
from Equation (\ref{eqn:141a}) one can derive a formula for 
moments -- this is precisely the $\bC$-valued case of the 
moment formula found in Proposition 4.30 of \cite{JeLi2019}, 
and is stated as follows.

\begin{corollary}   \label{cor:53}
For every $n \in \bN$, every colouring 
$c : \{ 1, \ldots , n \} \to \{ 1, \ldots , s \}$, and every 
$a_1 \in \cA_{c(1)}, \ldots a_n \in \cA_{c(n)}$, one has
\begin{equation}    \label{eqn:53a}
\varphi (a_1 \cdots a_n) = 
\sum_{  \begin{array}{c}
{\scriptstyle \pi \in NC(n;c)} \\
{\scriptstyle with \ VNRP}
\end{array} } \ \prod_{V \in \pi} 
\ \beta_{|V|} ( (a_1, \ldots , a_n) \mid V ).
\end{equation}
\end{corollary}

\begin{proof} Perform an additional summation over
interval partitions on both sides of 
Equation (\ref{eqn:141a}), and use the formula which
expresses moments in terms of Boolean cumulants.
\end{proof}

A basic fact concerning free cumulants is that 
$\kappa_n (a_1, \ldots , a_n) = 0$ whenever $n \geq 2$ and 
there exists an $m \in \{ 1, \ldots , n \}$ such that $a_m$ 
is a scalar multiple of the unit.  The next corollary gives 
the analogue of this fact when one uses Boolean cumulants.

\begin{corollary}    \label{cor:54}
Let $( \cA , \varphi )$ be a noncommutative probability
space and let $( \beta_n : \cA^n \to \bC )_{n=1}^{\infty}$ 
be the family of Boolean cumulant functionals associated to 
it.  Let $a_1, \ldots , a_n$ be elements of $\cA$, where 
$n \geq 2$, and suppose we are given an index 
$m \in \{ 1, \ldots , n \}$ for which it is known that 
$a_m = \oneA$.  Then
\begin{equation}   \label{eqn:54a}
\beta_n ( a_1, \ldots , a_n) = \left\{  
\begin{array}{ll}
0, & \mbox{ if $m=1$ or $m=n$,}            \\
\beta_{n-1} (a_1, \ldots , a_{m-1}, a_{m+1}, \ldots , a_n),
& \mbox{ if $1 < m < n$.}
\end{array}  \right.
\end{equation} 
\end{corollary}

\begin{proof}
Consider the unital subalgebras $\cA_1, \cA_2$ of $\cA$ defined 
by $\cA_1 = \cA$ and $\cA_2 = \bC \oneA$, and consider the 
colouring $c : \{ 1, \ldots , n \} \to \{ 1,2 \}$ defined by
\[
c(i) = \left\{  \begin{array}{ll}
1, & \mbox{ if $i \in \{ 1, \ldots , n \} \setminus \{ m \}$,} \\
2, & \mbox{ if $i = m$.}
\end{array}  \right.
\]
We then have $a_i \in \cA_{c(i)}$ for all $1 \leq i \leq n$, 
with $\cA_1$ being freely independent of $\cA_2$, hence
Theorem \ref{thm:141} applies to this situation and gives us 
a formula for $\beta_n (a_1, \ldots , a_n)$.  If $m=1$ or 
$m=n$, then $\beta_n (a_1, \ldots , a_n) = 0$ by 
Remark \ref{rem:52}.  If $1 < m < n$, then an immediate 
inspection shows that the only partition $\pi \in NC(n)$ 
with unique outer block and with VNRP is 
\[
\pi = \bigl\{ \, \{ 1, \ldots , m-1, m+1, \ldots , n \},
\, \{ m \} \, \bigr\} .
\]
Thus the sum on the right-hand side of (\ref{eqn:51a}) 
has in this case only one term, which is as indicated 
in (\ref{eqn:54a}) above (since 
$\beta_1 (a_m) = \beta_1 ( \oneA ) = 1$).
\end{proof}

$\ $

From Theorem \ref{thm:141} one can also get an explicit 
description of the Boolean cumulants of the sum of two 
free elements, as follows.

\begin{proposition}    \label{prop:55}
Let $( \cA , \varphi )$ be a noncommutative probability 
space and let $a,b \in \cA$ be such that $a$ is freely 
independent from $b$.  Consider the sequences of 
Boolean cumulants $( \beta_n (a) )_{n=1}^{\infty}$ and
$( \beta_n (b) )_{n=1}^{\infty}$ for $a$ and for $b$, 
respectively.  Then for every $n \geq 1$, the $n$-th 
Boolean cumulant of $a + b$ is 
\begin{equation}    \label{eqn:55a}
\beta_n (a+b) = 
\sum_{ \begin{array}{c}
{\scriptstyle \pi \in NC(n),} \\
{\scriptstyle \pi \ll 1_n}
\end{array} } 
\ \Bigl( \prod_{ \begin{array}{c}
{\scriptstyle U \in \pi, \ with}  \\
{\scriptstyle \depth_{\pi} (U) \ even} 
\end{array} } \
\beta_{|U|} (a)\Bigr)  \cdot
\Bigl( \prod_{ \begin{array}{c}
{\scriptstyle V \in \pi, \ with}  \\
{\scriptstyle \depth_{\pi} (V) \ odd} 
\end{array} } \ \beta_{|V|} (b) \Bigr) 
\end{equation}
\[ 
+ \ \sum_{ \begin{array}{c}
{\scriptstyle \pi \in NC(n),} \\
{\scriptstyle \pi \ll 1_n}
\end{array} } 
\ \Bigl( \prod_{ \begin{array}{c}
{\scriptstyle U \in \pi, \ with}  \\
{\scriptstyle \depth_{\pi} (U) \ even} 
\end{array} } \
\beta_{|U|} (b) \Bigr)  \cdot
\Bigl( \prod_{ \begin{array}{c}
{\scriptstyle V \in \pi, \ with}  \\
{\scriptstyle \depth_{\pi} (V) \ odd} 
\end{array} } \ \beta_{|V|} (b) \Bigr) .
\]
\end{proposition}

\begin{proof} 
We expand $\beta_n (a+b, \ldots , a+b)$ as a sum of 
$2^n$ terms by multilinearity, and then for each of the 
$2^n$ terms we use the formula for Boolean cumulants with 
free entries.  This takes us to an expression of
the form stipulated on the right-hand side of Equation 
(\ref{eqn:55a}) -- it is a large sum where every term of 
the sum is a product of Boolean cumulants of $a$ and of $b$.

We next make the following observation: for every 
$\pi \in NC(n)$ such that $\pi \ll 1_n$ there exist precisely 
two ways of colouring the blocks of $\pi$ in the colours 
``$a$'' and ``$b$'' such that VNRP holds; indeed, once we 
decide what is the colour of the unique outer block of $\pi$, 
everything else is determined.  By using this observation, 
we sort out the terms of the large sum indicated in the 
preceding paragraph, and we organize these terms into two 
separate sums, arriving precisely to the formula announced 
in (\ref{eqn:55a}).
\end{proof}

\begin{remark}   \label{rem:56}
The statement of Proposition \ref{prop:55} simplifies quite 
a bit when the two elements $a$ and $b$ have the same 
distribution.  In this case we only have one sequence 
$( \lambda_n )_{n=1}^{\infty}$, giving the Boolean cumulants 
for both $a$ and $b$, and Equation (\ref{eqn:55a}) becomes 
\begin{equation}   \label{eqn:56a}
\beta_n (a+b) = 
2 \cdot \sum_{ \begin{array}{c}
{\scriptstyle \pi \in NC(n),} \\
{\scriptstyle \pi \ll 1_n}
\end{array} } 
\ \Bigl( \prod_{ V \in \pi } \lambda_{|V|} \Bigr) .
\end{equation}
We note that Equation (\ref{eqn:56a}) can alternatively be 
obtained by using some known facts about the Boolean
Bercovici-Pata bijection introduced in \cite{BePa-Bi1999}.  
Indeed, let us apply this Bercovici-Pata bijection to the 
common distribution $\mu$ of $a$ and of $b$, and let us 
denote the resulting distribution by $\nu$.  (Here we view 
both $\mu$ and $\nu$ as linear functionals on $\bC [X]$, 
with $\nu$ being defined in terms of $\mu$ via the requirement 
that its free cumulants 
\footnote{ The free cumulant $\kappa_n ( \nu )$ is defined 
as the free cumulant $\kappa_n (X, \ldots , X)$ in the 
noncommutative probability space $( \bC [X], \nu )$, while 
the Boolean cumulant $\beta_n ( \mu )$ is defined as the 
Boolean cumulant $\beta_n (X, \ldots , X)$ in the 
noncommutative probability space $( \bC [X], \mu )$. }
satisfy 
$\kappa_n ( \nu ) = \beta_n ( \mu )$, $\forall \, n \in \bN$.)
It is known (cf. Theorem 1.2 in \cite{BeNi2008b}) 
that the Boolean cumulants of $\nu$ satisfy the relation
\begin{equation}   \label{eqn:56b}
\beta_n ( \nu ) = \frac{1}{2} \beta_n ( \mu \boxplus \mu ),
\ \ n \in \bN ,
\end{equation}
where $\mu \boxplus \mu$ (the ``free additive convolution''
of $\mu$ with itself) is the distribution of $a+b$.  On the
other hand, one can write 
\begin{equation}   \label{eqn:56c}
\beta_n ( \nu )
= \sum_{ \begin{array}{c}
{\scriptstyle \pi \in NC(n),}  \\
{\scriptstyle \pi \ll 1_n}
\end{array} }
\prod_{V \in \pi} \kappa_{|V|} ( \nu )                
= \sum_{ \begin{array}{c}
{\scriptstyle \pi \in NC(n),}  \\
{\scriptstyle \pi \ll 1_n}
\end{array} }
\prod_{V \in \pi} \beta_{|V|} ( \mu ),
\end{equation}
where the first equality in (\ref{eqn:56c}) is the identity 
expressing Boolean cumulants in terms of free cumulants. 
Putting together (\ref{eqn:56b}) and (\ref{eqn:56c}) leads to 
\begin{equation}   \label{eqn:56d}
\beta_n ( \mu \boxplus \mu ) =
2 \cdot \sum_{ \begin{array}{c}
{\scriptstyle \pi \in NC(n),} \\
{\scriptstyle \pi \ll 1_n}
\end{array} } 
\ \Bigl( \prod_{ V \in \pi } 
\beta_{|V|} ( \mu ) \Bigr) , \ \ n \in \bN ,
\end{equation}
which is a re-phrasing of (\ref{eqn:56a}).
\end{remark}
	
$\ $

\begin{remark}   \label{rem:57}
One can do a bit of further combinatorial analysis following 
to the statement of Proposition \ref{prop:55}, in order to go 
to the level of power series, and thus come up with some 
equations in $\eta$-series which can be used to obtain $\eta_{a+b}$.

More precisely, let us consider, same as in Proposition 
\ref{prop:55}, a noncommutative probability space 
$( \cA , \varphi )$, and let $a,b \in \cA$ be free.  
Remark \ref{rem:52} implies that the $\eta$-series 
$\eta_{a+b}(z) := \sum_{n=1}^{\infty} \beta_n(a+b)z^n$
splits as $\eta_{a+b}(z) =B_a(z)+B_b(z)$, where 
\begin{align*}
B_a(z)
&=\sum_{n=1}^{\infty} 
\Bigl( \sum_{c_2,\ldots,c_{n-1}\in\{a,b\}} 
\beta_n (a,c_2,\ldots,c_{n-1},a) \Bigr) z^n,    \\
B_b(z)
&=\sum_{n=1}^{\infty} 
\Bigl( \sum_{c_2,\ldots,c_{n-1}\in\{a,b\}} 
\beta_n (b,c_2,\ldots,c_{n-1},b) \Bigr) z^n .
\end{align*}
From Theorem \ref{thm:141} it follows that 
$\beta_n(a,c_2,\ldots,c_{n-1},a)$ is expressed as a sum over 
coloured non-crossing partitions with one outer block.  Upon 
sorting out terms according to what is this outer block,
one obtains the formula
\begin{align*}
B_a(z)&=\beta_1(a) z\\&+\beta_2(a)z^2\left(\sum_{m=0}^{\infty}
\left(\beta_1(b)z+\beta_2(b,b)z^2+\left(\beta_3(b,a,b)+\beta_3(b,b,b)\right)z^3
+ \ldots\right)^m \right)
\\&+\eta_3(a)z^2\left(\sum_{m=0}^{\infty}
\left(\eta_1(b)z+\eta_2(b,b)z^2+\left(\eta_3(b,a,b)+\eta_3(b,b,b)\right)z^3
+ \ldots\right)^m\right)^2 + \ldots
\end{align*}
After summing the geometric series that have appeared, and after doing
the similar calculation for $B_b (z)$, one arrives to the system 
of equations
\begin{equation}   \label{eqn:57a}
\left\{  \begin{array}{lll}
B_a(z) & = & \eta_a \left( \frac{z}{1-B_b(z)} \right) (1-B_b(z)),  \\
B_b(z) & = & \eta_b \left( \frac{z}{1-B_a(z)} \right) (1-B_a(z)).
\end{array}    \right.
\end{equation}
Solving the system (\ref{eqn:57a}) may be used as a path towards
the explicit calculation of the $\eta$-series
$\eta_{a+b}=B_a+B_b$.  This is of course not as smooth as using 
free cumulants and $R$-transforms (where one has the simplest 
possible formula, $R_{a+b} = R_a + R_b$), but we indicated how this
goes in anticipation of the analogous development for anticommutators
which we present in Section 6 below, and where free cumulants do not 
provide a simpler alternative.
\end{remark}

$\ $ 

\section{Boolean cumulants of products and of anticommutators}

This section is concerned with studying $*$--Boolean cumulants
of products of $*$--free random variables; as a byproduct of
that, we obtain the formula for Boolean cumulants of a free
anticommutator which was announced in Theorem \ref{thm:154} of
the Introduction.

For clarity, we start by discussing a concrete low-order example.

\begin{example}   \label{eg:61}
Assume that $a,b$ are $*$--free. We are interested to calculate 
\begin{equation}   \label{eqn:61a}
\beta_3 ( ab, ab, b^*a^* ) = ? \, ,
\end{equation}
where the expression sought on the right-hand side should be a
polynomial expression in the Boolean $*$--cumulants of 
$a$ and those of $b$.  We reach this goal in three steps.

\vspace{6pt}

{\em Step 1.} Use the formula for Boolean cumulants with products
as entries.

This step expresses $\beta_3 ( ab, ab, b^*a^* )$ as a sum of 
8 terms, indexed by the collection of partitions 
$\sigma \in \Int (6)$ with the property that
$\sigma \vee 
\{ \, \{ 1,2 \}, \, \{ 3,4 \}, \, \{ 5,6 \}\, \}
= 1_6$ (or, equivalently, that 
$\sigma \geq \{ \, \{ 1 \}, \, \{ 2,3 \}, \, \{ 4,5 \} \, \{ 6 \}  \, \}$).  
By keeping in mind the Remark \ref{rem:52}, we see 
that 5 of these terms are sure to be equal to $0$, and thus the 
conclusion of this step is that
\begin{equation}   \label{eqn:61b}
\beta_3 ( ab, ab, b^*a^* ) 
= \beta_1 (a) \beta_4 (b,a,b,b^{*}) \beta_1 (a^{*})
\end{equation}
\[
+ \beta_3 (a,b,a) \beta_2 (b,b^{*}) \beta_1 (a^{*})
+ \beta_6 (a,b,a,b,b^{*},a^{*}).
\]

\vspace{6pt}

{\em Step 2.} Use the formula for Boolean cumulants with free
entries.

This step takes on the Boolean cumulants which appear on the 
right-hand side of (\ref{eqn:61b}) and still mix together $a$'s 
and $b$'s -- we use Theorem \ref{thm:141} to express these 
Boolean cumulants as polynomials which separate the $a$'s 
from the $b$'s.  In order to process
$\beta_6 (a,b,a,b,b^{*},a^{*})$ we refer to Example \ref{ex:42}
and we also do the immediate calculation that
\[
\beta_4 (b,a,b,b^{*}) = \beta_1 (a) \beta_3 (b,b,b^{*}), 
\ \ \beta_3 (a,b,a) = \beta_2 (a,a) \beta_1 (b).
\]
The overall result of the calculation is that 
\begin{align}   \label{eqn:61c}
\beta_3 ( ab, ab, b^*a^* ) =
& \  \beta_1 (a) \beta_1 (a) \beta_1 (a^{*}) \cdot 
     \beta_3 (b,b,b^{*})
  + \beta_2 (a,a) \beta_1 (a^{*}) \cdot
    \beta_1 (b) \beta_2 (b,b^{*})                             \\
& + \beta_2(a,a^*)\beta_1(a) \cdot \beta_2(b,b)\beta_1(b^*)
  + \beta_2(a,a^*)\beta_1(a)\beta_3(b,b,b^*)     \nonumber    \\
&   \beta_3(a,a,a^*) \cdot \beta_1(b)\beta_1(b)\beta_1(b^*)
  + \beta_3(a,a,a^*)\beta_1(b)\beta_2(b,b^*).  \nonumber
\end{align}

The right-hand side of (\ref{eqn:61c}) is indeed a sum of 
products of $*$-cumulants of $a$ and of $b$, as we wanted.

\vspace{6pt}

{\em Step 3.} In order to understand what is going on, 
we record Equation (\ref{eqn:61c}) in the form:
\begin{equation}   \label{eqn:61d}
\beta_3 ( ab, ab, b^*a^* ) = 
\end{equation}
\[
\sum_{\pi \in \Pi} 
\prod_{ \begin{array}{c}
{\scriptstyle U \in \pi, } \\
{\scriptstyle of \ `colour' \ a} 
\end{array} } 
\ \beta_{|U|} ( (a,b,a,b,b^{*},a^{*}) \mid U) 
\cdot
\prod_{ \begin{array}{c}
{\scriptstyle V \in \pi, } \\
{\scriptstyle of \ `colour' \ b} 
\end{array} }
\ \beta_{|V|} ( (a,b,a,b,b^{*},a^{*}) \mid V),
\]
where $\Pi$ is the set of 6 non-crossing partitions depicted 
(in a way which includes colouring of blocks) in Figure 4.
The question then becomes: can we put into evidence the 
structure which underlies this special set $\Pi \subseteq NC(6)$?  
Upon staring a bit at Figure 4, it becomes quite appealing to 
believe that VRNP must be involved here!  This hunch is confirmed 
by Theorem \ref{thm:66} below, where the set $\Pi \subseteq NC(6)$ 
from Equation (\ref{eqn:61d}) becomes the correct indexing set 
corresponding to the tuple $\ee = (1,1,*) \in \{ 1,* \}^3$.
\end{example}

\vspace{6pt}

\begin{minipage}{1\linewidth}
\begin{center}
\begin{tikzpicture}			
			\draw[red,thick] (-8,0) -- (-8,0.25);
			\draw[red,thick] (-6,0) -- (-6,0.25);
			\draw[red,thick] (-3,0) -- (-3,0.25);
			\draw[blue,thick] (-7,0) to[out=90,in=90] node [sloped,above] {} (-5,0);
			\draw[blue,thick] (-5,0) to[out=90,in=90] node [sloped,above] {} (-4,0);
						
			\node[below,red] (1) at (-8,0) {$a$};
			\node[below,blue] (2) at (-7,0) {$b$};
			\node[below,red] (3) at (-6,0) {$a$};
			\node[below,blue] (4) at (-5,0) {$b$};
			\node[below,blue] (5) at (-4,0) {$b^*$};
			\node[below,red] (6) at (-3,0) {$a^*$};

			\draw[blue,thick] (1,0) -- (1,0.25);
			\draw[red,thick] (5,0) -- (5,0.25);
			\draw[red,thick] (0,0) to[out=90,in=90] node [sloped,above] {} (2,0);
			\draw[blue,thick] (3,0) to[out=90,in=90] node [sloped,above] {} (4,0);
			
			\node[below,red] (7) at (0,0) {$a$};
			\node[below,blue] (8) at (1,0) {$b$};
			\node[below,red] (9) at (2,0) {$a$};
			\node[below,blue] (10) at (3,0) {$b$};
			\node[below,blue] (11) at (4,0) {$b^*$};
			\node[below,red] (12) at (5,0) {$a^*$};
			\end{tikzpicture}
		\end{center}
	\end{minipage}
	\begin{minipage}{1\linewidth}
		\begin{center}
			\begin{tikzpicture}
				\draw[red,thick] (-6,0) -- (-6,0.25);
				\draw[blue,thick] (-4,0) -- (-4,0.25);
				\draw[red,thick] (-8,0) to[out=90,in=90] node [sloped,above] {} (-3,0);
				\draw[blue,thick] (-7,0) to[out=90,in=90] node [sloped,above] {} (-5,0);
				
				\node[below,red] (1) at (-8,0) {$a$};
				\node[below,blue] (2) at (-7,0) {$b$};
				\node[below,red] (3) at (-6,0) {$a$};
				\node[below,blue] (4) at (-5,0) {$b$};
				\node[below,blue] (5) at (-4,0) {$b^*$};
				\node[below,red] (6) at (-3,0) {$a^*$};

				\draw[red,thick] (2,0) -- (2,0.25);
				\draw[red,thick] (0,0) to[out=90,in=90] node [sloped,above] {} (5,0);
				\draw[blue,thick] (1,0) to[out=90,in=90] node [sloped,above] {} (3,0);
				\draw[blue,thick] (3,0) to[out=90,in=90] node [sloped,above] {} (4,0);

				\node[below,red] (7) at (0,0) {$a$};
				\node[below,blue] (8) at (1,0) {$b$};
				\node[below,red] (9) at (2,0) {$a$};
				\node[below,blue] (10) at (3,0) {$b$};
				\node[below,blue] (11) at (4,0) {$b^*$};
				\node[below,red] (12) at (5,0) {$a^*$};
			\end{tikzpicture}
		\end{center}
	\end{minipage}
	\begin{minipage}{1\linewidth}
	\begin{center}
		\begin{tikzpicture}
		\draw[blue,thick] (-4,0) -- (-4,0.25);
		\draw[blue,thick] (-7,0) -- (-7,0.25);
		\draw[blue,thick] (-5,0) -- (-5,0.25);
		\draw[red,thick] (-8,0) to[out=90,in=90] node [sloped,above] {} (-6,0);
		\draw[red,thick] (-6,0) to[out=90,in=90] node [sloped,above] {} (-3,0);
		
		\node[below,red] (1) at (-8,0) {$a$};
		\node[below,blue] (2) at (-7,0) {$b$};
		\node[below,red] (3) at (-6,0) {$a$};
		\node[below,blue] (4) at (-5,0) {$b$};
		\node[below,blue] (5) at (-4,0) {$b^*$};
		\node[below,red] (6) at (-3,0) {$a^*$};

		\draw[blue,thick] (1,0) -- (1,0.25);
		\draw[red,thick] (0,0) to[out=90,in=90] node [sloped,above] {} (2,0);
		\draw[red,thick] (2,0) to[out=90,in=90] node [sloped,above] {} (5,0);
		\draw[blue,thick] (3,0) to[out=90,in=90] node [sloped,above] {} (4,0);

		\node[below,red] (7) at (0,0) {$a$};
		\node[below,blue] (8) at (1,0) {$b$};
		\node[below,red] (9) at (2,0) {$a$};
		\node[below,blue] (10) at (3,0) {$b$};
		\node[below,blue] (11) at (4,0) {$b^*$};
		\node[below,red] (12) at (5,0) {$a^*$};
		\end{tikzpicture}
        \end{center}
        \end{minipage}

\begin{center}
{\bf Figure 4.}
{\em The set $\Pi \subseteq NC(6)$ used in 
Equation (\ref{eqn:61d}).}
\end{center}

$\ $

\begin{remark}   \label{rem:62}
A similar calculation to the one shown in 
Example \ref{eg:61} could be done in the world of free cumulants, 
and would lead to a similarly looking formula which expresses 
the free cumulant $\kappa_3 (ab, ab, b^{*}a^{*} )$ in terms 
of the $*$-free cumulants of $a$ and the $*$-free cumulants 
of $b$.  The formula with free cumulants is ``just a bit'' 
more complicated than the one for Boolean cumulants obtained
in (\ref{eqn:61c}) above: it has 7 terms instead of 6, where 
6 of the 7 terms are having the same structure as in 
(\ref{eqn:61c}), while the 7th term is
\[
\kappa_2(a,a^*)\kappa_1(a)\kappa_1(b)\kappa_2(b,b^*),
\]
corresponding to the partition 
$\{\{1,6\},\{2\},\{3\},\{4,5\}\} \in NC(6)$.
But, unlike the partitions corresponding to the other 6 terms,
this latter partition does not satisfy VNRP! (It is one of the 
two partitions without VNRP that were featured in Figure 3(b) 
in Section 4.)
\end{remark}

$\ $

We now move to discuss Boolean cumulants for general words made 
with $ab$ and with $(ab)^{*}$.  To this end, we recall from the 
Introduction the terminology about ac-friendly partitions and 
about the canonical alternating colouring of a non-crossing 
partition.  

In the Introduction, the definition of what is an 
ac-friendly partition in $NC(2n)$ was made by only 
referring to depths of blocks.  It will come in handy to 
note that this notion can also be approached via canonical 
alternating colourings, as explained in the next proposition.

\begin{proposition}    \label{prop:63}
Let $n$ be a positive integer and let $\pi$ be a 
partition in $NC(2n)$.  Then $\pi$ belongs to the set
$\NCacfriendly (2n)$ introduced in Definition \ref{def:151}
if and only if it satisfies the following two conditions:

\vspace{6pt}

\em{(AC-Friendly1)} $\OuterMax ( \pi ) \subseteq 
\{ 1,3, \ldots , 2n-1 \} \cup \{ 2n \}$.

\vspace{6pt}

(AC-Friendly2') $\calt_{\pi} (2i-1) \neq \calt_{\pi} (2i),
\ \ \forall \, 1 \leq i \leq n$,
where $\calt_{\pi} : \{ 1, \ldots , 2n \} \to \{ 1,2 \}$
is the canonical alternating colouring associated to $\pi$.
\end{proposition}

\begin{proof} ``$\Rightarrow$'' We assume that 
(AC-Friendly1) and (AC-Friendly2) hold.  Suppose, toward a contradiction, that (AC-Friendly2') fails. Then there exists an $i$ such that $\calt_\pi(2i-1)\not=\calt_\pi(2i)$.  This can only happen if $2i-1$ and $2i$ belong to distinct blocks $A$ and $B$ that are ``siblings," i.e., have same parent $C$.  But then, for $j=2i-1$, we have that $\depth_\pi(j)=\depth\pi(C)+1=\depth_\pi(j+1)$ contradicting (AC-Friendly2).

\vspace{6pt}

``$\Leftarrow$'' We assume that (AC-Friendly1) and (AC-Friendly2') hold.  
Suppose, toward a contradiction, that (AC-Friendly2) fails.  
Let $j=2i-1\not\in\OuterMax(\pi)$ be such that $\depth_\pi(j)=\depth_\pi(j+1)$.  
Since $j$ and $j+1$ cannot belong to distinct outer blocks (as then we 
would have $j\in\OuterMax(\pi)$) this  can only happen if they either 
belong to the same block or they belong to blocks that have the same parent.  
In both of these cases we then have that $\calt_\pi(2i-1)=\calt_\pi(2i)$,
contradicting (AC-Friendly2').
\end{proof}

\begin{notation}  \label{def:64}
Let $n$ be a positive integer, let $\pi$ be a partition
in $\NCacfriendly (2n)$, and consider the canonical 
alternating colouring 
$\calt_{\pi} : \{ 1, \ldots , 2n \} \to \{ 1,2 \}$.
We use the values 
$\calt_{\pi} (1), \calt_{\pi} (3), \ldots , \calt_{\pi} (2n-1)$
in order to create a tuple $\ee \in \{ 1,* \}^n$, as follows:
\begin{equation}   \label{eqn:64a}
\ee(i) = \left\{  \begin{array}{ll}
1,  &  \mbox{ if $\calt_{\pi} (2i-1) = 1$, }   \\
*,  &  \mbox{ if $\calt_{\pi} (2i-1) = 2$  }
\end{array}  \right\} \ , 1 \leq i \leq n.
\end{equation}
The tuple $\ee \in \{ 1,* \}^n$ so defined will be denoted 
as $\oddtuple ( \pi )$.
\end{notation}

\begin{remark}    \label{rem:65}
Let $n$ be a positive integer and let $\pi$ be a partition
in $\NCacfriendly (2n)$.  Knowing what is the tuple 
$\oddtuple ( \pi ) \in \{ 1,* \}^n$ gives us precisely the 
same information as if we knew what is the colouring
$\calt_{\pi} : \{ 1, \ldots , 2n \} \to \{ 1,2 \}$.  Indeed, 
if we know $\oddtuple ( \pi )$ then we can use Equation 
(\ref{eqn:64a}) to find the values 
$\calt_{\pi} (1), \calt_{\pi} (3), \ldots , \calt_{\pi} (2n-1)$, 
after which we can also find out what are 
$\calt_{\pi} (2), \calt_{\pi} (4), \ldots , \calt_{\pi} (2n)$
based on the fact that $\calt_{\pi} (2i) \neq \calt_{\pi} (2i-1)$,
$1 \leq i \leq n$.
\end{remark}

We can now state the desired formula about the joint 
Boolean cumulants of $ab$ and $ba$.

\begin{theorem}   \label{thm:66}
Let $( \cA , \varphi )$ be a $*$-probability space 
and let $( \beta_n : \cA^n \to \bC )_{n=1}^{\infty}$ be the 
family of Boolean cumulant functionals associated to it.
We consider two selfadjoint elements $a,b \in \cA$ such that 
$a$ is freely independent from $b$, and we consider the 
sequences of Boolean cumulants
$( \beta_n (a) )_{n=1}^{\infty}$ and 
$( \beta_n (b) )_{n=1}^{\infty}$.

\vspace{6pt}

(1) Let $n$ be a positive integer and let
$\ee = ( \ee (1), \ldots , \ee (n)) \in \{ 1,* \}^n$
be such that $\ee (1) = 1$.  One has
\begin{equation}   \label{eqn:66a}
\beta_n \bigl( \, (ab)^{\ee (1)}, \ldots ,
(ab)^{\ee (n)} \, \bigr) =
\end{equation}
\[
\sum_{ \begin{array}{c}
{\scriptstyle \pi \in \NCacfriendly (2n),} \\
{\scriptstyle such \ that}    \\
{\scriptstyle \oddtuple ( \pi ) = \ee}
\end{array} } 
\ \Bigl( 
\prod_{ \begin{array}{c}
{\scriptstyle U \in \pi, \ with}  \\
{\scriptstyle \calt_{\pi} (U) = 1} 
\end{array} } \
\beta_{|U|} (a) \Bigr)  \cdot
\Bigl( \prod_{ \begin{array}{c}
{\scriptstyle V \in \pi, \ with}  \\
{\scriptstyle \calt_{\pi} (V) = 2} 
\end{array} } \ \beta_{|V|} (b) \Bigr). 
\]

\vspace{6pt}

(2) Let $n$ be a positive integer and let
$\ee = ( \ee (1), \ldots , \ee (n)) \in \{ 1,* \}^n$
be such that $\ee (1) = *$.  Consider the complementary 
tuple $\ee ' \in \{ 1,* \}^n$, uniquely determined by 
the requirement that $\ee ' (i) \neq \ee (i)$, for all
$1 \leq i \leq n$.  One has
\begin{equation}   \label{eqn:66b}
\beta_n \bigl( \, (ab)^{\ee (1)}, \ldots ,
(ab)^{\ee (n)} \, \bigr) =
\end{equation}
\[
\sum_{ \begin{array}{c}
{\scriptstyle \pi \in \NCacfriendly (2n),} \\
{\scriptstyle such \ that}    \\
{\scriptstyle \oddtuple ( \pi ) = \ee '}
\end{array} } 
\ \Bigl(   \prod_{ \begin{array}{c}
{\scriptstyle U \in \pi, \ with}  \\
{\scriptstyle \calt_{\pi} (U) = 1} 
\end{array} } \
\beta_{|U|} (b) \Bigr)  \cdot
\Bigl(   \prod_{ \begin{array}{c}
{\scriptstyle V \in \pi, \ with}  \\
{\scriptstyle \calt_{\pi} (V) = 2} 
\end{array} } \ \beta_{|V|} (a) \Bigr). 
\]
\end{theorem}

\begin{proof}  We will assume the case (1), i.e., that $\varepsilon(1)=1$ (the case $\varepsilon(1)=*$ is the `mirror' image of this case and is left to the reader).  Note that
$\varepsilon$ induces a colouring $c\colon\{1,\ldots,2n\}\to\{1,2\}$ by $c(2i-1)=1, c(2i)=2$ when $\varepsilon(i)=1$ and
$c(2i-1)=2$, $c(2i)=1$ when $\varepsilon(i)=*$; in other words, $c$ assigns $1$ to positions where there is an $a$ in $(ab)^{\varepsilon(1)}\ldots(ab)^{\varepsilon(n)}$ and $2$ to positions where there is a $b$.

We do again the steps presented 
in Example \ref{eg:61}, where the discussion is now 
made to go in reference to an abstract tuple $\ee$,
rather than the special case of 
$(1,1,*) \in \{ 1,* \}^3$.  

Step 1 takes us to a sum over interval partitions 
partitions $\sigma = \{ J_1, \ldots , J_p \} \in \Int (2n)$
such that $\sigma\vee\{\{1\},\{2,3\},\ldots, \{2n-2,2n-1\},\{2n\}\}=1_{2n}$.  The condition 
$\sigma\vee\{\{1\},\{2,3\},\ldots, \{2n-2,2n-1\},\{2n\}\}=1_{2n}$ is equivalent to saying that
$\max (J_k) \in \{ 1,3, \ldots , 2n-1 \}$ for all
$1 \leq k < p$.  The latter  can be expressed 
directly in terms of the cardinalities of the blocks of 
$\sigma$: either $\sigma = 1_{2n}$, or it has 
$|J_1|, |J_p|$ odd and $|J_k|$ even for all $1 < k < p$.  This implies, among other things, that for all $k<p$ we have that
$c(\max(J_k))\not=c(\min(J_{k+1}))$.

The observation in Remark \ref{rem:52} that certain mixed Boolean 
cumulants with free entries have to vanish, then yields that 
we are left to only consider the 
cases where for each $k$ we have that $c(\min(J_k))=c(\max(J_k))$.  

Step 2: We now apply the separation formula from 
Theorem \ref{thm:141} to each block $J_k$ to get that $\beta_n \bigl( \, (ab)^{\ee (1)}, \ldots ,
(ab)^{\ee (n)} \, \bigr)$ is equal to
\[
\sum_{\begin{array}{c} \scriptstyle{\sigma=\{J_1,\ldots,J_p\}\in \Int(2n)}\\
 \scriptstyle{ \sigma\vee \{\{1,2\},\ldots,\{2n-1,2n\}\}=1_{2n}}\\
 \scriptstyle{\forall k, c(\min(J_k))=c(\max(J_k))} \end{array}} 
\sum_{ \begin{array}{c}
\scriptstyle{\pi \in NC(2n,c),} \\
\scriptstyle{\pi\ll \sigma} \\
\scriptstyle{\pi\mbox{ has VNRP }}
\end{array} } 
\ \Bigl( 
\prod_{ \begin{array}{c}
{\scriptstyle U \in \pi, \ with}  \\
{\scriptstyle c(U) = 1} 
\end{array} } \
\beta_{|U|} (a) \Bigr)  \cdot
\Bigl( \prod_{ \begin{array}{c}
{\scriptstyle V \in \pi, \ with}  \\
{\scriptstyle c(V) = 2} 
\end{array} } \ \beta_{|V|} (b) \Bigr). 
\]
Step 3: We claim that partitions $\pi$ appearing in the summation above are precisely those 
in $\NCacfriendly (2n)$ for which $\oddtuple ( \pi ) = \ee$ and that we always have
$\calt_\pi = c$; from whence the formula (\ref{eqn:66a}) follows.  

Since consecutive blocks of $\sigma$ start with distinct colours we have that the consecutive outer blocks of $\pi$ have distinct colours.  When we combine this observation with the facts that the first block of $\pi$ has colour $1$ and that $\pi$ has VNRP (with respect to $c$) we get that $c=\calt_\pi$.

As already mentioned above we have that the condition $\pi\ll\sigma$, $\sigma\in\Int(2n)$, 
$\sigma\vee \{\{1,2\},\ldots,\{2n-1,2n\}\}=1_{2n}$ is equivalent to $\pi$ satisfying (AC-Friendly1). 

It now remains to prove that every $\pi$ in the above summation formula satisfies 
(AC-Friendly2).  Suppose the converse.  Then there is an odd number 
$j\not\in\OuterMax(\pi)$ such that $\depth_\pi(j)=\depth_\pi(j+1)$.  
Since $j$ and $j+1$ have distinct $c$-colours we have 
that $j$ and $j+1$ belong to distinct blocks.  They cannot both belong to outer 
blocks (in this case it would follow that $j\in\OuterMax(\pi)$).  The equality 
$\depth_\pi(j)=\depth_\pi(j+1)$ then implies that the blocks containing $j$ and $j+1$ 
must have the same parent, and considering the $c$-colour of the parent-block leads to
an immediate contradiction with VNRP.
\end{proof}

\begin{remark}   \label{rem:67}
We note that Theorem \ref{thm:66} could be stated in 
the framework where $( \cA , \varphi )$ is a plain
noncommutative probability space, and we look at joint 
Boolean cumulants of $ab$ and $ba$, with $a$ free from $b$.

The statement of Theorem \ref{thm:66} 
could also be extended to the case when we deal with joint Boolean 
cumulants of $ab$ and $(ab)^{*}$ in a $*$-probability space, 
without assuming that $a$ and $b$ are selfadjoint.  The proof 
would be the same, only that it would result in stuffier formulas.
\end{remark}

\begin{remark}   \label{rem:58}
The special case $\ee = (1,1, \ldots , 1)$ of Theorem 
\ref{thm:66} gives a formula for the Boolean cumulant 
$\beta_n (ab,ab, \ldots , ab)$.  We explain here that 
this is precisely the formula from \cite{BeNi2008} which was
reviewed in Equation (\ref{eqn:12a}) of the Introduction.

To this end, let us first note that if a partition 
$\sigma \in \NCacfriendly (2n)$ has 
$\oddtuple ( \sigma ) = (1,1, \ldots , 1)$, then the 
canonical alternating colouring $\calt_{\sigma}$ must have 
$\calt_{\sigma} (2i-1) = 1$ and $\calt_{\sigma} (2i) = 2$ 
for all $1 \leq i \leq n$.  So $\calt_{\sigma}$ is
the colouring of $\{ 1, \ldots , 2n \}$ by parity, which 
implies that every block of $\sigma$ is contained either
in $\{ 1,3, \ldots , 2n-1 \}$ or in 
$\{ 2,4, \ldots , 2n \}$.  We note moreover that such 
$\sigma$ is sure to only have two outer blocks, the blocks 
$W'$ and $W''$ which contain the numbers $1$ and $2n$, respectively.
Indeed, if $\sigma$ had an outer block $W \neq W', W''$, then the 
condition (AC-Friendly1) satisfied by $\sigma$ would imply that 
$\min (W)$ is an even number and $\max (W)$ is an odd number 
-- not possible!  Knowing these things about $\sigma$, plus the 
fact that $\sigma$ has VNRP, leads to the conclusion that $\sigma$
must be of the form 
$\sigma = \piodd \sqcup ( \Kr_n ( \pi ) )^{\mathrm{(even)}}$
for some $\pi \in NC(n)$; this is precisely the content of 
Lemma 6.8 in \cite{BeNi2008}.

Conversely, let $\pi$ be in $NC(n)$ and consider the partition 
$\sigma = \piodd \sqcup ( \Kr_n ( \pi ) )^{\mathrm{(even)}}
\in NC(2n)$.  Then Lemma 6.6 of \cite{BeNi2008} assures us that
the canonical alternating colouring of $\sigma$ is the 
colouring of $\{ 1, \ldots , 2n \}$ by parity.  The same lemma of
\cite{BeNi2008} also records the fact that $\sigma$ has exactly two
outer blocks, the ones containing the numbers $1$ and $2n$, and
this clearly entails the condition (AC-Friendly1) from
Proposition \ref{prop:63} above.  In view of Proposition 
\ref{prop:63}, we then conclude that $\sigma \in \NCacfriendly (2n)$ 
and at the same time we see that the tuple $\oddtuple ( \sigma )$ 
is equal to $(1,1, \ldots , 1)$.

The discussion from the preceding two paragraphs shows that 
when we make $\ee = (1,1, \ldots , 1)$ in Theorem \ref{thm:66}(1),
the summation on the right-hand side of Equation (\ref{eqn:66a})
is made over partitions of the form 
$\piodd \sqcup ( \Kr_n ( \pi ) )^{\mathrm{(even)}}$,
with $\pi$ running in $NC(n)$.  Moreover, when looking at the 
term of the summation which is indexed by a $\pi \in NC(n)$, one 
sees that the two products appearing there are precisely 
$\prod_{U \in \pi} \beta_{|U|} (a)$ and 
$\prod_{V \in \Kr_n ( \pi )} \beta_{|V|} (b)$.  Hence this 
special case of Equation (\ref{eqn:66a}) retrieves Equation 
(\ref{eqn:12a}), as claimed at the beginning of the remark.
\end{remark}

$\ $

By starting from Theorem \ref{thm:66}, it is easy to prove
the formula announced in Theorem \ref{thm:154} of the 
Introduction (and repeated below), concerning 
the Boolean cumulants of a free anticommutator.

\begin{theorem}    \label{thm:610}
Let $( \cA , \varphi )$ be a 
noncommutative probability space and let $a,b \in \cA$
be such that $a$ is freely independent from $b$.
Consider the sequences of Boolean cumulants 
$( \beta_n (a) )_{n=1}^{\infty}$ and 
$( \beta_n (b) )_{n=1}^{\infty}$ for $a$ and for $b$,
respectively.  Then, for every $n \geq 1$, the $n$-th 
Boolean cumulant of $ab + ba$ is 
\[
\beta_n (ab+ba) = 
\sum_{\pi \in \NCacfriendly (2n)} \ 
\Bigl( \prod_{ U \in \pi, \calt_{\pi} (U) = 1} 
\beta_{|U|} (a) \Bigr)  \cdot
\Bigl( \prod_{ V \in \pi, \calt_{\pi} (V) = 2} 
\beta_{|V|} (b) \Bigr)  
\]
\[ 
+ \ \sum_{\pi \in \NCacfriendly (2n)} \ 
\Bigl( \prod_{ U \in \pi, \calt_{\pi} (U) = 1} 
\beta_{|U|} (b) \Bigr)  \cdot
\Bigl( \prod_{ V \in \pi, \calt_{\pi} (V) = 2} 
\beta_{|V|} (a) \Bigr) . 
\]
\end{theorem}

\begin{proof}  We write that 
\begin{equation}   \label{eqn:610a}
\beta_n ( ab+ba , \ldots , ab+ba )
= \sum_{\ee \in \{ 1,* \}^n } \,
\beta_n \bigl( \, (ab)^{\ee (1)}, \ldots , 
                  (ab)^{\ee (n)} \, \bigr) = 
\end{equation}
\[
\sum_{\ee \in \{ 1,* \}^n, \ee(1)=1 } \beta_n \bigl( \, (ab)^{\ee (1)}, \ldots , 
                  (ab)^{\ee (n)} \, \bigr) + \sum_{\ee \in \{ 1,* \}^n, \ee(1)=* } \beta_n \bigl( \, (ab)^{\ee (1)}, \ldots , 
                  (ab)^{\ee (n)} \, \bigr). 
\]
Note that every $\pi\in\NCacfriendly (2n)$ determines a unique $\ee:=\oddtuple(\pi)\in\{1,*\}^n$ such that $\ee(1)=1$.  Also recall that every $\ee\in \{1,*\}^n$ for which $\ee(1)=*$ we have $\ee'\in\{1,*\}^n$ determined by $\ee'(k)\not=\ee(k)$ for all $k$ (in particular $\ee'(1)=1$). We invoke the formulas found in Theorem 
\ref{thm:66} to get
\begin{align*}
&\sum_{\ee \in \{ 1,* \}^n, \ee(1)=1 } \beta_n \bigl( \, (ab)^{\ee (1)}, \ldots , 
                  (ab)^{\ee (n)} \, \bigr)\\ 
&= \sum_{\ee \in \{ 1,* \}^n, \ee(1)=1 }\sum_{ \begin{array}{c}
{\scriptstyle \pi \in \NCacfriendly (2n),} \\
{\scriptstyle such \ that}    \\
{\scriptstyle \oddtuple ( \pi ) = \ee}
\end{array} } 
\ \Bigl( 
\prod_{ \begin{array}{c}
{\scriptstyle U \in \pi, \ with}  \\
{\scriptstyle \calt_{\pi} (U) = 1} 
\end{array} } \
\beta_{|U|} (a) \Bigr)  \cdot
\Bigl( \prod_{ \begin{array}{c}
{\scriptstyle V \in \pi, \ with}  \\
{\scriptstyle \calt_{\pi} (V) = 2} 
\end{array} } \ \beta_{|V|} (b) \Bigr) \\
&= \sum_{
{\scriptstyle \pi \in \NCacfriendly (2n)} } 
\ \Bigl( 
\prod_{ \begin{array}{c}
{\scriptstyle U \in \pi, \ with}  \\
{\scriptstyle \calt_{\pi} (U) = 1} 
\end{array} } \
\beta_{|U|} (a) \Bigr)  \cdot
\Bigl( \prod_{ \begin{array}{c}
{\scriptstyle V \in \pi, \ with}  \\
{\scriptstyle \calt_{\pi} (V) = 2} 
\end{array} } \ \beta_{|V|} (b) \Bigr),
\end{align*}
and 
\begin{align*}
&\sum_{\ee \in \{ 1,* \}^n, \ee(1)=* } \beta_n \bigl( \, (ab)^{\ee (1)}, \ldots , 
                  (ab)^{\ee (n)} \, \bigr)\\ 
&= \sum_{\ee \in \{ 1,* \}^n, \ee(1)=* }\sum_{ \begin{array}{c}
{\scriptstyle \pi \in \NCacfriendly (2n),} \\
{\scriptstyle such \ that}    \\
{\scriptstyle \oddtuple ( \pi ) = \ee'}
\end{array} } 
\ \Bigl( 
\prod_{ \begin{array}{c}
{\scriptstyle U \in \pi, \ with}  \\
{\scriptstyle \calt_{\pi} (U) = 1} 
\end{array} } \
\beta_{|U|} (b) \Bigr)  \cdot
\Bigl( \prod_{ \begin{array}{c}
{\scriptstyle V \in \pi, \ with}  \\
{\scriptstyle \calt_{\pi} (V) = 2} 
\end{array} } \ \beta_{|V|} (a) \Bigr)\\
&= \sum_{
{\scriptstyle \pi \in \NCacfriendly (2n)} } 
\ \Bigl( 
\prod_{ \begin{array}{c}
{\scriptstyle U \in \pi, \ with}  \\
{\scriptstyle \calt_{\pi} (U) = 1} 
\end{array} } \
\beta_{|U|} (b) \Bigr)  \cdot
\Bigl( \prod_{ \begin{array}{c}
{\scriptstyle V \in \pi, \ with}  \\
{\scriptstyle \calt_{\pi} (V) = 2} 
\end{array} } \ \beta_{|V|} (a) \Bigr).
\end{align*}
\end{proof}

$\ $

\section{On the $\eta$--series of a free anticommutator}

In this section we show how observations about Boolean cumulants of a
free anticommuator from the previous section can be captured in the 
form of a system of equation at the level of $\eta$-series.

\subsection{Equations in power series.}

$\ $

\noindent
Throughout this subsection we fix a $*$-probability space 
$( \cA , \varphi )$ and two elements $a,b \in \cA$. 
For clarity of arguments it is better if at first we do not 
assume that $a$ and $b$ are selfadjoint, and we discuss the formal 
power series 
$\eta_{ab,b^*a^*}\in\mathbb{C}\langle\langle 
z_a,z_{a^*},z_b,z_{b^*}\rangle\rangle$ 
defined as
\begin{equation}  \label{eqn:6a}
\eta_{ab,b^*a^*} =
\sum_{n=1}^{\infty} \sum_{ (\ee (1), \ldots , \ee (n)) \in \{1,* \}^n} 
\ \beta_n ( (ab)^{\ee (1)}, \ldots , (ab)^{\ee (n)}) 
(z_a z_b)^{\ee (1)} \cdots  (z_a z_b)^{\ee (n)},
\end{equation}
where on the right-hand side of (\ref{eqn:6a}) we make the convention
to put $(z_a z_b)^{*} = z_{b^{*}} z_{a^{*}}$ (so, for instance, the 
term corresponding to $\ee = (1,1,*) \in \{ 1,* \}^3$ is 
$\beta_3 (ab, ab, b^{*}a^{*}) z_a z_b z_a z_b z_{b^{*}} z_{a^{*}}$).
At some point down the line we will however switch to the special case 
when $a=a^*,b=b^*$ and $z_a= z_{{a}^*} = z_b = z_{{b}^*} =: z$;
in this special case the series from Equation (\ref{eqn:6a}) becomes 
a series of one variable, which is nothing but $\eta_{ab+ba}(z^2)$.

Returning to Equation (\ref{eqn:6a}) we observe that 
$\eta_{ab,b^*a^*}$ splits naturally as a sum,
\begin{equation}   \label{eqn:6c}
\eta_{ab,b^*a^*} = \eta^{1,1} + \eta^{1,*}
+ \eta^{*,1} + \eta^{*,*},
\end{equation}
where $\eta^{1,1}$ contains those terms of $\eta_{ab,b^*a^*}$ which 
correspond to Boolean cumulants beginning and ending with $ab$, 
$\eta^{1,*}$ contains those terms of $\eta_{ab,b^*a^*}$ which correspond 
to Boolean cumulants that begin with $ab$ and end with $b^*a^*$, etc.
Under the assumption that $\{ a,a^{*} \}$ is free from $\{ b, b^{*} \}$,
one can then make a number of structural observations about the four power 
series introduced in Equation (\ref{eqn:6c}).

Let us start with $\eta^{1,1}$.  Every term of this series is 
of the form 
\begin{equation}    \label{eqn:6d}
\beta_n(ab,\ldots,ab)z_a z_b\dotsm z_a z_b . 
\end{equation}
To the Boolean cumulant appearing in (\ref{eqn:6d}) we apply the 
formula for Boolean cumulants with products as entries from Proposition 
\ref{prop:210}, together with the property that a Boolean cumulant 
which starts with $a$ or $a^*$ and ends with $b$ or $b^*$ (or vice versa) 
vanishes, as noticed in Remark \ref{rem:52}.
Then it follows from Theorem \ref{thm:610} 
that the cumulant from (\ref{eqn:6d}) is a sum of terms of the form
\begin{align}\label{eqn:92}
\beta_{l_0}(a,b,\ldots,a)\left(\prod_{j=1}^{n-1}
\beta_{k_{j}}(b,\ldots,b^*)\beta_{l_{j}}(a^*,\ldots,a)\right)
\beta_{k_{n}}(b,\ldots,a,b),
\end{align}
for $n\geq 1$, $l_0,k_n\geq 1$ and $l_j,k_j\geq 2$ for $j=1,\ldots,n-1$.

For $l,l^\prime\in\{a,a^*\}$ or $l,l^\prime\in\{b,b^*\}$ we define power series $f_{l,l^*}\in \mathbb{C}\langle\langle z_a,z_{a^*},z_b,z_{b^*}\rangle\rangle$ as power series which contain Boolean cumulants starting with $l$ and ending with $l^\prime$ which can appear in the expression above. 
To be more precise, consider $\eta_{a,a^*,b,b^*}\in\mathbb{C}\langle\langle z_a,z_{a^*},z_b,z_{b^*}\rangle\rangle$ the joint $\eta$--series of $a,a^*,b,b^*$ then $f_{l,l^\prime}$ is a restriction of $\eta_{a,a^*,b,b^*}$ to these terms which begin with $l$, end with $l^\prime$, and corresponding word $z_{l}\dotsm z_{l^\prime}$ is a subword of some word of the type $(z_az_b)^{l_1}(z_{b^*}z_{a^*})^{k_1}\dotsm(z_az_b)^{l_n}(z_{b^*}z_{a^*})^{k_n}$ with $n\geq 1$ and $l_1,k_n\geq 0$ and $k_1,\ldots,k_{n-1},l_2,\ldots,l_n\geq 1$. 
We have for example
\begin{equation}   \label{eqn:6e}
\left\{   \begin{array}{lll}
f_{a,a} &=& \beta_1(a)z_a+\beta_3(a,b,a)z_az_bz_{a^*}
           +\beta_5(a,b,b^*,a^*,a)z_az_bz_{b^*}z_{a^*}z_a+\ldots \\
f_{a^*,a} &=& \beta_2(a^*,a)z_{a^*} z_a+
\beta_4(a^*,a,b,a)z_{a^*}z_az_bz_a+\ldots
\end{array}  \right.
\end{equation}
With such notation, Equation \eqref{eqn:92} can be written as
\begin{align*}
\eta^{1,1}=f_{a,a}(1-f_{b,b^*} f_{a^*,a})^{-1} f_{b,b},
\end{align*}
with the convention
\begin{align*}
(1-f_{b,b^*} f_{a^*,a})^{-1}=\sum_{n=0}^{\infty}\left(f_{b,b^*} f_{a^*,a}\right)^{n}.
\end{align*}

Similar analysis to the one done above for $\eta^{1,1}$ can be done 
for the remaining three power series from the right-hand side of
(\ref{eqn:6c}).  We have for example
\begin{align*}
\eta^{1,*}=f_{a,a^*}+f_{a,a} (1-f_{b,b^*}f_{a^*,a})^{-1} f_{b,b^*} f_{a^*,a^*},
\end{align*}
where the additional term $f_{a,a^*}$ comes from the term of the 
expansion of $\beta_n (ab,\ldots,b^*a^*)$ with one outer block.
 
The system of equations which comes out of the preceding discussion can be nicely 
written in matrix form, as follows:
\begin{align}  \label{eqeta}
\begin{bmatrix}
\eta^{1,1} & \eta^{1,*} \\ 
\eta^{*,1} & \eta^{*,*}
\end{bmatrix}=\begin{bmatrix}
f_{aa} (1-f_{bb^*}f_{a^*a})^{-1}f_{bb} & f_{aa^*}+f_{aa} (1-f_{bb^*}f_{a^*a})^{-1} f_{bb^*} f_{a^*a^*}\\
f_{b^*b}+f_{b^*b^*} (1-f_{a^*a} f_{bb^*})^{-1} f_{a^*a} f_{bb} & f_{b^*b^*} (1-f_{a^*a} f_{bb^*})^{-1} f_{a^*a^*}
\end{bmatrix}.
\end{align}

Once that Equation (\ref{eqeta}) is put into evidence, the problem of computing 
the $\eta$-series $\eta_{ab, b^{*}a^{*}}$ from (\ref{eqn:6a}) is reduced 
to the one of computing the power series $f_{l,l^\prime}$.  
We will take on this job in the special case when
$a$ and $b$ are assumed to be (freely independent and) selfadjoint. 
In this special case, 
the determination of the series $f_{l, l^\prime}$ has to be made 
in terms of the Boolean cumulants of $a$ and of $b$, or equivalently,
in terms of the $\eta$-series of these elements.  The mechanism for doing
so is provided by the following theorem.

\begin{theorem}\label{thm:61}
Notation as above, where we assume that $a,b$ are selfadjoint and freely 
independent, and we put $z_a = z_{a^{*}} = z_b = z_{b^{*}} =: z$.  Define
\begin{align*}
	F_a=\begin{bmatrix}
		f_{a,a} & f_{a,a^*} \\
		f_{a^*,a} & f_{a^*,a^*}
	\end{bmatrix},\quad
	F_b=\begin{bmatrix}
	f_{b,b} & f_{b,b^*} \\
	f_{b^*,b} & f_{b^*,b^*}
	\end{bmatrix}
	\end{align*}
	and
	\begin{align*}
		H_a&=\begin{bmatrix}
		f_{bb} (1-f_{b^*b})^{-1} & f_{b,b^*}+f_{b,b} (1-f_{b^*,b})^{-1} f_{b^*,b^*}\\ (1-f_{b^*,b})^{-1} & (1-f_{b^*,b})^{-1} f_{b^*,b^*} 
		\end{bmatrix}, \\
		H_b&=\begin{bmatrix}
		 (1-f_{a,a^*})^{-1}f_{a,a} & (1-f_{a,a^*})^{-1} \\f_{a^*,a}+f_{a^*,a^*} (1-f_{a,a^*})^{-1} f_{a,a} & f_{a^*,a^*}(1-f_{a,a^*})^{-1}  
		\end{bmatrix}.
	\end{align*}
Then one has
	\begin{align}\label{eqn:acPowerSeries}
	\begin{cases}
		F_aH_a&=\eta_a(z H_a),\\
		F_bH_b&=\eta_b(z H_b).
	\end{cases}
	\end{align}
where $\eta_a$ and $\eta_b$ are the $\eta$-series of $a$ and of $b$
(as reviewed in Notation \ref{def:27b}). 
\end{theorem}

\begin{proof}
The main tool we will use in order to establish the relations stated in
\eqref{eqn:acPowerSeries} is the VNRP property from Theorem \ref{thm:141}.
We will only prove the first of the two relations, as the proof of the 
second one is analogous.  Note that, even though our current hypotheses 
are such that $(ab)^*=b^*a^*=ba$, we will nevertheless continue to allow 
for occurrences of $a^*$ and $b^{*}$, which we will use to distinguish 
between $a$'s and $b$'s coming from a product $ab$ versus $a$'s and $b$'s 
coming from a product $ba$.
	
For $l,l^\prime\in\{a,a^*\}$ consider $f_{l,l^\prime}$. Each term of $f_{l,l^\prime}$ is a joint 
Boolean cumulant of $a,a^*,b,b^*$ which starts with $l$ and ends with $l^\prime$. 
According to Theorem \ref{thm:141} all partitions in the expansion have 
a unique outer block.  We will sort the terms of $f_{l,l^\prime}$ according to 
the structure of this outer block. So let us fix a possibility for what the
outer block could be -- this has to start with $l$, has to end with 
$l^\prime$, and must contain only $a$'s and/or $a^*$'s. By $h_{k,k^\prime}$ 
we will denote the power series which occurs between consecutive $k$ 
and $k^\prime$ in the outer block. Then we have 
\begin{align*}
f_{l,l^\prime}=\sum_{n=1}^{\infty}
\sum_{\substack{ w(1),\ldots,w(n) \in \{a,a^*\}\\ 
                 with \ w(1)=l,w(n)=l^\prime}}
\beta_n\left(w(1),\ldots,w(n)\right)
z_{w(1)}h_{w(1),w(2)} z_{w(2)} \cdots   
\end{align*}
\[
\hfill
\cdots z_{w(n-1)} h_{w(n-1),w(n)} z_{w(n)}.
\]
	For example we have 
	\begin{align*}
	f_{a,a}=\beta_1(z) z_a+\beta_2(a,a)z_a h_{a,a} z_a+\beta_3(a,a^*,a)z_a h_{a,a^*} z_{a^*} h_{a^*,a} z_a+\ldots
	\end{align*}
Consider then the matrix $\widetilde{H}_a$ defined by
	\begin{align*}
\widetilde{H}_a := \begin{bmatrix}
		h_{a,a} & h_{a,a^*} \\
		h_{a^*,a} & h_{a^*,a^*}
		\end{bmatrix}.
	\end{align*}
Since we assume $a=a^*$ the relation between $f_{l,l^\prime}$ and $h_{k,k^\prime}$ 
can be written on the level of $2 \times 2$ matrices in the form
	\begin{align*}
		F_a=\sum_{n=1}^{\infty}\beta_n(a)\begin{bmatrix}
		z_a&0\\0&z_{a^*}
		\end{bmatrix}\left(\widetilde{H}_a \begin{bmatrix}
		z_a&0\\0&z_{a^*}
		\end{bmatrix}\right)^{n-1}.
	\end{align*}
We assumed that $z_a=z_{a^*}=z$, thus multiplying both sides by 
$\widetilde{H}_a$ immediately gives
	\begin{align*}
	F_a \widetilde{H}_a&=\eta_a(z \widetilde{H}_a).
	\end{align*}
We are left to prove that the matrix $\widetilde{H}_a$ is in fact the 
same as the $H_a$ defined in the statement of the theorem, i.e. that 
the series $h_{l,l'}$ which have appeared in the discussion are such 
that
\begin{align*}
\begin{bmatrix}
h_{a,a} & h_{a,a^*} \\
h_{a^*,a} & h_{a^*,a^*}
\end{bmatrix}
		&=\begin{bmatrix}
f_{b,b} (1-f_{b^*,b})^{-1} & f_{b,b^*}+f_{b,b} (1-f_{b^*,b})^{-1} f_{b^*,b^*}\\ 
(1-f_{b^*,b})^{-1} & (1-f_{b^*,b})^{-1} f_{b^*,b^*} 
		\end{bmatrix}.
\end{align*}
We will analyze separately each of the four entries of this matrix 
equality.  The observation that is used in the discussion of each of 
the four entries is as follows: due to VNRP, it is immediate that between 
any two consecutive elements of the outer block we will have products of 
Boolean cumulants starting with $b$ or $b^*$ and ending with $b$ or $b^*$. 

	\begin{enumerate}

\item {\em Entry (1,1).} 
Assume that the two consecutive variables in the outer block are $a$ and $a$. 
Observe that the element coming right after the first $a$ must be a $b$. The 
element appearing immediately to the left of the second $a$ from the outer 
block could be $a^*$ or $b$, but if it was $a^*$, then by VNRP this $a^*$ 
would be in the outer block and thus the consecutive elements considered in 
the outer block would be $a$ and $a^*$ (rather than $a$ and $a$). We thus 
conclude that the element appearing immediately to the left of the second 
$a$ from the outer block is a $b$.  In order to not violate VNRP between $a$ 
and $a$ we can get a cumulant $\beta_{n_0}(b,\ldots,b)$ or for some $k>1$ we 
can get $\beta_{n_0}(b,\ldots,b)\left(\prod_{i=0}^k\beta_{n_i}(b^*,\ldots,b)\right)$ 
for $n_i\geq 2$.  In $\eta_{n_0}(b,\ldots,b)$ and $\eta_{n_i}(a,\ldots,b^*)$ there 
are $a$ and $a^*$ as arguments but by Theorem \ref{thm:141} writing the term as 
a joint cumulant gives exactly all terms with VNRP property. We conclude that 
in each $a,a$ pocket we can get any term of the power series $f_{bb} (1-f_{b^*b})^{-1}$

\item  {\em Entry (1,2).}
Assume that the two consecutive variables in the outer block are $a$ and $a^*$. 
In this case we find, by a similar argument as above, that between $a$ and $a^*$ 
one can get one block of the form $\beta_{m}(b,\ldots,b^*)$ for $m\geq 2$ or if 
there are more blocks they are of the form 
$\beta_{n_0}(b,\ldots,b)\prod_{i=0}^k\beta_{n_i}(b^*,\ldots,b)
\beta_{n_{k+1}}(b^*,\ldots,b^*)$ for $k\geq 0$ and $n_0,n_{k+1}\geq 1$ and 
$n_i\geq 2$ for $i=1,\ldots,k$. Thus in each $a, a^*$ pocket we get the power 
series $f_{b,b} (1-f_{b^*,b})^{-1} f_{b^*,b^*}$

\item {\em Entry (2,1).} 
Assume that the two consecutive variables in the outer block are $b^*$ and $b$. Then similar analysis shows that in each pocket we can get $\left(1-f_{b^*, b}\right)^{-1}$.

\item {\em Entry (2,2).}
Assume that the two consecutive variables in the outer block are $a^*$ and $a^*$. Then similar analysis shows that in each pocket we can get $\left(1-f_{b^*, b}\right)^{-1}f_{b^*, b^*}$.

\end{enumerate}
\end{proof}

\begin{remark}
(1) Suppose that $( \cA , \varphi )$ is a $C^{*}$-probability space.
In this case the $\eta$--series of $a$ and $b$ are convergent power series 
around zero and the system of equations \eqref{eqn:acPowerSeries} can be 
solved for analytic functions in some neighbourhood of zero.

(2) By specializing to the case when $a$ and $b$ have the same 
distribution, one immediately obtains Theorem \ref{thm:161} of
the Introduction.  Stated in a bit more detail, this goes as 
follows.
\end{remark}

\begin{corollary}  \label{cor:62}
Consider the setting of Theorem \ref{thm:61}, where we make the 
additional assumption that $a$ and $b$ have the same distribution. 
Then one has 
\begin{align*}
	&f_{b,b}=f_{a^*,a^*},\quad f_{b,b^*}=f_{a^*,a},\\
	&f_{b^*,b}=f_{a,a^*},\quad f_{b^*,b^*}=f_{a,a}.
\end{align*}
The system of equations \eqref{eqn:acPowerSeries} reduces to a 
single equation,
\begin{align}\label{eqn:acSameDistr}
	F_aH_a&=\eta_a(z H_a),
\end{align}
where
\begin{align*}
H_a&=\begin{bmatrix}
f_{a^*,a^*} (1-f_{a,a^*})^{-1} & f_{a^*,a}+f_{a^*,a^*} (1-f_{a,a^*})^{-1} f_{a,a}\\ 
(1-f_{a,a^*})^{-1} & (1-f_{a,a^*})^{-1} f_{a,a} 
\end{bmatrix}.
\end{align*}
Moreover, in this case the formula for the $\eta$--series of $ab+ba$ also 
simplifies, and from Equation \eqref{eqeta} one gets that
\begin{align}\label{eqn:etaSamDist}
	\eta_{ab+ba}(z^2)=2\left(f_{a,a^*}(z)+\frac{f_{a,a}(z)f_{a^*,a^*}(z)}{1-f_{a^*,a}(z)}\right).
\end{align}
$\ $ \hfill $\square$
\end{corollary}

\vspace{10pt}

\subsection{The special case of symmetric distributions.}

$\ $

\noindent
As explained in Remark \ref{rem:156}, the calculation of the distribution 
of a free anticommutator $ab+ba$ is much more approachable in the case 
when $a$ and $b$ have symmetric distributions.  This fact also manifests 
itself in the framework of Theorem \ref{thm:61}, where the hypothesis 
that $a$ and $b$ have symmetric distributions leads to the simplified 
statement of the next proposition.

\begin{proposition}   \label{prop:64}
Consider the framework and notation of Theorem \ref{thm:61}, and
let us also assume that $a$ and $b$ have symmetric distributions.
Then the power series appearing in Theorem \ref{thm:61} are such that
$f_{a,a}=f_{a^*,a^*}=f_{b,b}=f_{b^*,b^*}=0$, and the matrices $H_a$ and $H_b$ 
from that theorem become:
\[
		H_a = \begin{bmatrix}
		0 & f_{b,b^*} \\ (1-f_{b^*,b})^{-1} & 0 
		\end{bmatrix}, \ \ \  
	H_b = \begin{bmatrix}
	0 & (1-f_{a,a^*})^{-1} \\ 
        f_{a^*,a} & 0
	\end{bmatrix} .
\]
The system of equations \eqref{eqn:acPowerSeries} simplifies to
		\begin{align}\label{eqn:SymAcPowerSeries}
		\begin{cases}
		\begin{bmatrix}
		f_{a,a^*}(1-f_{b^*,b})^{-1} & 0 \\0 & f_{a^*,a}f_{b,b^*}
		\end{bmatrix}&=\eta_a(z H_a),\\ \\
		\begin{bmatrix}
		f_{b,b^*}f_{a^*,a} & 0 \\0 & (1-f_{a,a^*})^{-1}f_{b^*,b}
		\end{bmatrix}&=\eta_b(z H_b),
		\end{cases}
		\end{align}
and one has $\eta_{ab+ba}(z^2)=f_{a,a^*}(z)+f_{b^*,b}(z)$.
\end{proposition}

\begin{proof}
We first note that from the fact that 
$\varphi (a^{2n-1}) = \varphi (b^{2n-1})=0$ for all $n \in \bN$
and from the formulas connecting Boolean cumulants to moments we get 
$\beta_{2n-1}(a)=\beta_{2n-1}(b)=0$ for all $n \in \bN$.
Now, the coefficients of the power series 
$f_{a,a},f_{a^*,a^*},f_{b,b},f_{b^*,b^*}$ are odd length joint Boolean 
cumulants of $a$ and $b$, thus each of them contains either an odd number 
of $a$'s or an odd number of $b$'s.  From Theorem \ref{thm:141} one 
gets that every such joint cumulant is zero (since upon writing it as 
a sum of products in the way indicated by Theorem \ref{thm:141}, each 
term will contain an odd length Boolean cumulant of $a$ or of $b$).
Thus we get $f_{a,a}=f_{a^*,a^*}=f_{b,b}=f_{b^*,b^*}=0$, and the claims 
of the proposition then follow from the general formulas obtained in 
Theorem \ref{thm:61}.
\end{proof}

\begin{corollary}    \label{cor:65}
In the framework of Proposition \ref{prop:64}, assume moreover that 
$a$ and $b$ have the same distribution.  Then the Equations 
\eqref{eqn:SymAcPowerSeries} of Proposition \ref{prop:64} further 
simplify to
		\begin{align}\label{eqn:AcSymSameDistr}
		\begin{bmatrix}
			f_{a,a^*}(1-f_{a,a^*})^{-1} & 0 \\0 & f_{a^*,a}^2
		\end{bmatrix}&=\eta_a(z H_a),
		\end{align}
		where
		\begin{align*}
		H_a&=\begin{bmatrix}
		0 & f_{a^*,a}\\ (1-f_{a,a^*})^{-1} & 0 
		\end{bmatrix}.
		\end{align*}
Moreover, in this case one gets that
$\eta_{ab+ba}(z^2)=2f_{a,a^*}(z)$.
\hfill $\square$
\end{corollary}

$\ $

While the examples of free anticommutators of symmetric distributions 
can be handled with the methods from \cite{NiSp1998}, we nevertheless 
discuss two such examples, mostly in order to point out that 
Theorem \ref{thm:154} can be used (in some sense) in reverse, for 
getting corollaries about the enumeration of ac-friendly
non-crossing partitions.

\begin{example}   \label{ex:66}
Suppose that $p$ and $q$ are two free projections in a $*$-probability space 
$( \cA , \varphi )$, such that $\varphi (p) = \varphi (q) = 1/2$, and
let $a := 2p-1, b = 2q-1$.  Then $a$ and $b$ are as in Corollary \ref{cor:65},
where the common distribution of $a$ and $b$ is the symmetric Bernoulli 
distribution $\frac{1}{2} ( \delta_{-1} + \delta_1 )$.  The distribution of 
$ab+ba$ can be computed by using Corollary \ref{cor:65}; but this special 
example is actually much easier to handle, since it is immediate that 
$u := ab$ is a Haar unitary in $( \cA , \varphi )$, which implies by direct
calculation (cf. Example 1.14 in Lecture 1 of \cite{NiSp2006}) that 
$ab + ba = u + u^{*}$ has the arcsine distribution with density 
$( \pi \sqrt{4-t^2} )^{-1}$ on the interval $[-2,2]$.

On the other hand, one can also look at what Theorem \ref{thm:154}
has to say in connection to this example, and this leads to the 
corollary stated next.  For this corollary, recall that a 
non-crossing partition $\sigma \in NC(2n)$ is said to be a pairing 
when every block $V$ of $\sigma$ has $|V| = 2$.  The set of all 
non-crossing pairings in $NC(2n)$ is denoted by $NC_2 (2n)$.  
It is well-known that the cardinality of $NC_2 (2n)$ is equal to 
$\Cat_n$, the same Catalan number which counts all the non-crossing 
partitions in $NC(n)$.
\end{example}

\begin{corollary}  \label{cor:67}
For every $m \in \bN$, one has that 
$| \, \NCacfriendly (4m) \cap NC_2 (4m) \, | = \Cat_{m-1}$
and that
$\NCacfriendly (4m-2) \cap NC_2 (4m-2) = \emptyset$.
\end{corollary}

\begin{proof}
Take $a$ and $b$ as in the preceding example.  It is an easy 
exercise to check that the common Boolean cumulants 
$( \lambda_n )_{n=1}^{\infty}$ for $a$ and for $b$ come out 
as $\lambda_2 = 1$ and $\lambda_n = 0$ for all $n \neq 2$.
Equation (\ref{eqn:155a}) from Remark \ref{rem:155}(2) thus 
tells us that 
\[
\beta_n (ab+ba) = 2 \cdot 
| \, \NCacfriendly (2n) \cap NC_2 (2n) \, |,
\ \ \forall \, n \in \bN .
\]
On the other hand, the direct calculation of the $\eta$-series
of the arcsine distribution gives us that 
$\beta_n (ab+ba) = 2  \, \Cat_{(n-2)/2}$  when $n$ is even,
and that $\beta_n (ab+ba) = 0$ when $n$ is odd, which
leads to the formulas stated in the corollary.
\end{proof}

\begin{remark}  \label{rem:68}
Let $m$ be a positive integer.  It is easy to easy that if 
$\pi \in NC_2 (2m)$ and if $\pi$ has $\{ 1, 2m \}$ as a pair, 
then the natural process of ``doubling'' the pairs of $\pi$ 
leads to a pairing in $\NCacfriendly (4m)$.  This construction 
produces $\Cat_{m-1}$ examples of pairings in $\NCacfriendly (4m)$, 
and the above corollary tells us that all the pairings in 
$\NCacfriendly (4m)$ are obtained in this way.
\end{remark}

\begin{example}   \label{ex:69}
Suppose we want to repeat the trick from Corollary \ref{cor:67} 
in order to calculate the number of partitions 
$\sigma \in \NCacfriendly (2n)$ with the property that every block 
$V \in \sigma$ has even cardinality.
To this end we now start with two selfadjoint elements 
$a, b$ in a $*$-probability space $( \cA , \varphi )$ such 
that $a$ is free from $b$ and such that both $a$ and $b$ 
have distribution 
\begin{equation}   \label{eqn:69a}
\frac{1}{4}\delta_{-\sqrt{2}}+\frac{1}{2}\delta_{0}+
\frac{1}{4}\delta_{\sqrt{2}} 
\end{equation}
with respect to $\varphi$.  The reason for choosing to use 
the distribution in (\ref{eqn:69a}) is that its $\eta$-series is 
$z^2 / (1-z^2)$, which makes the common sequence 
$( \lambda_n )_{n=1}^{\infty}$ of Boolean cumulants for $a$ and 
for $b$ to be given by 
\begin{equation}   \label{eqn:69b}
\lambda_n = \left\{  \begin{array}{ll}
1, & \mbox{if $n$ is even,}  \\
0, & \mbox{if $n$ is odd.}
\end{array}  \right.
\end{equation}

When applying Corollary \ref{cor:65} to the $a$ and $b$ of the 
present example, the system of equations presented in 
\eqref{eqn:AcSymSameDistr} becomes
	\begin{align*}
	\begin{cases}
		f_{a,a^*}(1-z^2f_{a^*,a}(1-f_{a,a^*})^{-1})&=z^2 f_{a^*,a}\\
		f_{a^*,a}(1-z^2f_{a^*,a}(1-f_{a,a^*})^{-1})&=z^2 (1-f_{a,a^*})^{-1}.
	\end{cases}
	\end{align*}
Upon some further processing, we find that the series $f_{a,a^{*}}$ satisfies
the equation
\begin{align*}
	f_{a,a^*}(z)(1-f_{a,a^*}(z))^3=z^4.
\end{align*}
Lagrange inversion formula gives 
\begin{align*}
f_{a,a^*}(z) =
\sum_{n=1}^\infty \frac{3}{4n-1}  {4n-1 \choose n-1} z^{4n},
\end{align*}
and (in view of the formula $\eta_{ab+ba}(z^2)=2f_{a,a^*}(z)$ 
from Corollary \ref{cor:65}) we come to the conclusion that the 
$\eta$-series of $ab+ba$ is
\begin{equation}   \label{eqn:69c}
\eta_{ab+ba}(z) =
2 \sum_{n=1}^\infty \frac{3}{4n-1} {4n-1 \choose n-1} z^{2n}.
\end{equation}

When we look at what Theorem \ref{thm:154} has to say in this 
particular case, we obtain the corollary stated next.  In the
corollary we will use the notation
\[
NC^{ (\mathrm{even}) } (2n)
:= \{ \sigma \in NC (2n) \mid 
\mbox{ every block $V \in \sigma$ has even cardinality} \}.
\]
\end{example}

\begin{corollary}   \label{cor:610}
For every $m \in \bN$, one has that 
\[
| \, \NCacfriendly (4m) \cap NC^{(\mathrm{even})} (4m) \, | 
= \frac{3}{4m-1} {4m-1 \choose m-1} 
\]
and that
$\NCacfriendly (4m-2) \cap NC^{(\mathrm{even})} (4m-2) 
 = \emptyset$. 
\end{corollary}

\begin{proof}
Take $a$ and $b$ as in the preceding example. 
Equation (\ref{eqn:155a}) from Remark \ref{rem:155}(2) 
(used in conjunction to the formula for $\lambda_n$'s 
found in (\ref{eqn:69b})) tells us that 
\[
\beta_n (ab+ba) = 2 \cdot 
| \, \NCacfriendly (2n) \cap NC^{(\mathrm{even})} (2n) \, |,
\ \ \forall \, n \in \bN .
\]
On the other hand, $\beta_n (ab+ba)$ is obtained by extracting
the coefficient of $z^n$ in the equality of power series which 
appeared in (\ref{eqn:69c}).  This immediately leads to the 
formulas stated in the corollary.
\end{proof}

\vspace{10pt}

\subsection{A non-symmetric example.} 

$\ $

\noindent
In this subsection we look at the example where $a$ and $b$ 
have distribution $\frac{1}{2} (\delta_0+ \delta_2)$.

This example offers a very good illustration of how one gets 
to have different distributions for the free commutator and 
anticommutator.  The commutator $i(ab-ba)$ has exactly the 
same arcsine distribution as in Example \ref{ex:66}; indeed, 
the $a,b$ of the current example are obtained by adding $1$ 
to the $a,b$ of Example \ref{ex:66}, and the translation by 
$1$ does not affect the commutator.  The anti-commutator 
$ab+ba$ turns out to have a different distribution, as 
stated in the next proposition.  (Recall that the graph of 
the density $f(x)$ indicated in the proposition was shown 
in Figure 2 at the end of the Introduction, together with 
a histogram of eigenvalues of random matrix approximation.)

\begin{proposition}   \label{prop:611}
Notations as above, with $a,b$ free and having distribution 
$\frac{1}{2} (\delta_0+ \delta_2)$.  Then the distribution of 
$ab+ba$ is the absolutely continuous measure on the interval 
$[-1,8]$ which has density $f(x)$ described as follows:
\begin{align*}
f(x)=\begin{cases}
& \frac{\sqrt{2}}{\pi}\frac{\sqrt{-1-\sqrt{\frac{x-8}{x}}-\frac{4}{x}}}{8-3\sqrt{(x-8 )x}-x}
  \quad\quad\mbox{ for } x \in (-1,0), \\
&                                      \\
& \frac{1}{\pi}\frac{\sqrt{x(4\sqrt{1+x}-x-4)}
  +3\sqrt{(8-x)(4\sqrt{1+x}+x+4)}
  -8\sqrt{\frac{4\sqrt{1+x}-4-x}{x}}}{8(8-x)(1+x)}
  \quad\mbox{for } x \in{(0,8).}
\end{cases}
\end{align*}
\end{proposition}

\begin{proof}
We have $\eta_a (z) = z/(1-z)$, hence Equation \eqref{eqn:acSameDistr} 
amounts here to
\begin{align}\label{eqn:acBinomial}
F_a=z (1-z H_a)^{-1},
\end{align}
where $(1-z H_a)^{-1}$ can be written explicitly as 
\begin{align*}
\frac{1}{z^2f_{a^*,a}
+z(f_{a,a}+f_{a^*,a^*})+f_{a,a^*}-1}
\begin{bmatrix}
z f_{a,a}+f_{a,a^*}-1 & z f_{a^*,a}(f_{a,a^*}-1)-z f_{a,a} f_{a^*,a^*} \\ 
-z & z f_{a^*,a^*}+f_{a,a^*}-1 
\end{bmatrix}.
\end{align*}
When solving Equation \eqref{eqn:acBinomial}, one gets 6 solutions. 
However, when plugged into the formula for $\eta_{ab+ba}(z^2)$ 
given in Corollary \ref{cor:62}, only one of the 6 solutions satisfies 
the condition that $\eta_{ab+ba}(z^2)$ has a double zero at $z=0$. 
After substituting $z$ by $\sqrt{z}$ in this solution, we come to 
the conclusion that
\begin{equation}   \label{eqn:611x}
\eta_{ab+ba}(z) = 
1 - \sqrt{ (1-8z) \frac{1-2z-\sqrt{1-8z}}{2z} }.
\end{equation}
From the latter formula for the $\eta$-series, a routine calculation 
takes us to the Cauchy transform of the distribution of $ab+ba$, which 
is 
\begin{align*}
G_{ab+ba}(z)=\frac{\sqrt{2}\sqrt{1+\sqrt{\frac{z-8}{z}}
+\frac{4}{z}}}{8+3\sqrt{(z-8)z}-z}.
\end{align*}
Finally, by using the Stieltjes inversion formula on
$G_{ab+ba} (z)$, we find the form of the density $f(x)$ which 
was stated in the proposition.
\end{proof}

Same as the examples from the preceding subsection,
the example considered here has a combinatorial significance,
and can be used to infer the formula indicated in Equation 
(\ref{eqn:15a}) of the Introduction for the generating series 
of cardinalities of sets of ac-friendly partitions.

\begin{corollary}    \label{cor:612}
The generating function for cardinalities of sets 
$\NCacfriendly (2n)$ is
\[
\sum_{n=1}^{\infty}
| \, \NCacfriendly (2n) \, | z^n 
= \frac{1}{2}
  - \sqrt{ (1-8z) \frac{1-2z-\sqrt{1-8z}}{8z} }.
\]
\end{corollary}

\begin{proof}
All the Boolean cumulants of $a$ and of $b$ in this example are 
equal to $1$.  As a consequence of this, Equation (\ref{eqn:155a}) 
from Remark \ref{rem:155}(2) simply says that
\[
\beta_n (ab+ba) = 2 \cdot | \, \NCacfriendly (2n) \, |,
\ \ \forall \, n \in \bN .
\]
The generating function for the cardinalities of $\NCacfriendly(2n)$
is thus equal to $\eta_{ab+ba} (z) /2$.  We substitute this into the 
formula for $\eta_{ab+ba} (z)$ which appeared in 
Equation (\ref{eqn:611x}) during the proof of the preceding proposition, 
and the corollary follows.
\end{proof}

$\ $


\begin{thebibliography}{99}

\bibitem{BeNi2008} S.T. Belinschi, A. Nica.
$\eta$-series and a Boolean Bercovici-Pata bijection for bounded
$k$-tuples, {\em Advances in Mathematics} 217 (2008), 1-41.

\bibitem{BeNi2008b} S.T. Belinschi, A. Nica.
On a remarkable semigroup of homomorphisms with respect to 
free multiplicative convolution,
{\em Indiana University Mathematics Journal} 57 (2008), 
1679-1713.

\bibitem{BePa-Bi1999} H. Bercovici, V. Pata.
Stable laws and domains of attraction in free probability theory.
With an appendix by P. Biane: The density of free stable 
distributions,
{\em Annals of Mathematics} 149 (1999), 1023-1060.
 
\bibitem{HeMaSp2018} J.W. Helton, T. Mai, R. Speicher.
Applications of realizations (aka linearizations) to 
free probability, 
{\em Journal of Functional Analysis} 274 (2018), 1-79.

\bibitem{JeLi2019} D. Jekel, W. Liu.
An operad of non-commutative independences defined by trees,
preprint 2019 (arXiv:1901.09158).

\bibitem{KrSp2000} B. Krawczyk, R. Speicher.
Combinatorics of free cumulants, 
{\em Journal of Combinatorial Theory, Series A} 
90 (2000), 267-292.

\bibitem{Kr1972} G. Kreweras.
Sur les partitions non-croise\'es d'un cycle,
{\em Discrete Mathematics} 1 (1972), 333-350.

\bibitem{NiSp-Vo1996} A. Nica, R. Speicher. 
On the multiplication of free $n$-tuples of
non-commutative random variables.  With an appendix by 
D. Voiculescu: Alternative proofs for the type II free
Poisson variables and for the compression results,
{\em American Journal of Mathematics} 118 (1996), 
799-837.

\bibitem{NiSp1997} A. Nica, R. Speicher.
A Fourier transform for multiplicative functions on 
non-crossing partitions, 
{\em Journal of Algebraic Combinatorics} 6 (1997),
141-160.

\bibitem{NiSp1998} A. Nica, R. Speicher.
Commutators of free random variables,
{\em Duke Mathematical Journal} 92 (1998), 553-592.

\bibitem{NiSp2006} A. Nica, R. Speicher.
{\em Lectures on the combinatorics of free probability},
London Mathematical Society Lecture Note Series 335,
Cambridge University Press, 2006.

\bibitem{PoWa2011} M. Popa, J.-C. Wang.
On multiplicative conditionally free convolution,
{\em Transactions of the American Mathematical Society}
363 (2011), 6309-6335.

\bibitem{St1997} R.P. Stanley.
{\em Enumerative combinatorics}, Volume 1,
Cambridge University Press, 1997.

\bibitem{Va2003} V. Vasilchuk.
On the asymptotic distribution of the commutator and
anticommutator of random matrices,
{\em Journal of Mathematical Physics} 44 (2003),
1882-1908.

\bibitem{Vo1987} D. Voiculescu.
Multiplication of certain non-commutative random variables,
{\em Journal of Operator Theory} 18 (1987), 223-235.
\end{thebibliography}
\end{document}